\documentclass[final]{siamltex}
\usepackage[]{graphicx}
\pdfoutput=1
\usepackage{color}
\usepackage{multicol}
\usepackage{verbatim}
\usepackage{multirow}
\usepackage{amsmath,amsfonts}

\newcommand{\defeq}{\> \lower 2pt \hbox{$\buildrel{\rm def} \over = $}\>}
\newtheorem{example}[theorem]{Example}
\newtheorem{remark}{Remark}

\title{ Competition between transients in the rate of approach to a fixed point}

\author{Judy Day\thanks{Mathematical Biosciences Institute, The Ohio State University, 1735 Neil Ave, Jennings Hall, Columbus, Ohio, 43210 ({\tt jday@mbi.osu.edu}).}
        \and Jonathan Rubin\thanks{Department of Mathematics, University of Pittsburgh, 301 Thackeray Hall, Pittsburgh, Pennsylvania, 15260 (rubin@math.pitt.edu).} \and Carson C. Chow\thanks{Laboratory of Biological Modeling, NIDDK, National Institutes of Health, Building 12A, Room 4007, 12 South Drive MSC 5621, Bethesda, Maryland, 20892 (carsonc@mail.nih.gov).}}

\begin{document}
\maketitle

\begin{abstract}
Dynamical systems studies of differential equations often focus on the
behavior of solutions near critical points and on invariant manifolds,
to elucidate the organization of the associated flow.  In addition,
effective methods, such as the use of Poincar\'{e} maps and phase
resetting curves, have been developed for the study of periodic
orbits.  However, the analysis of transient dynamics associated with
solutions {\em on their way} to an attracting fixed point has not
received much rigorous attention.  This paper introduces methods for
the study of such transient dynamics.  In particular, we focus on the
analysis of whether one component of a solution to a system of
differential equations can overtake the corresponding component of a
reference solution, given that both solutions approach the same stable
node.  We call this phenomenon {\em tolerance}, which derives from a certain biological effect.
Here, we establish certain general conditions, based on the initial
conditions associated with the two solutions and the properties of the
vector field, that guarantee that tolerance does or does not occur in
two-dimensional systems.  We illustrate these conditions in
particular examples, and we derive and demonstrate additional
techniques that can be used on a case by case basis to check for
tolerance.  Finally, we give a full rigorous analysis of tolerance in
two-dimensional linear systems.

\end{abstract}



\pagestyle{myheadings}
\thispagestyle{plain}
\markboth{J. DAY, J. RUBIN AND C.C. CHOW}{COMPETITION IN RATE OF APPROACH TO A FIXED POINT}

\section{Introduction}
\label{Section: Introduction}

Relative to asymptotic behavior, transients have received little attention in the study of nonlinear dynamical systems. For example, how the rate of approach to a stable fixed point, away from the asymptotic limit, is affected by the choice of initial conditions within the basin of attraction of that fixed point has not to our knowledge been well characterized. In this work, we consider a comparison of the transient dynamics of pairs of trajectories with similar asymptotic behaviors.
The motivation for this work arises from a biological phenomenon known as tolerance, which refers to a reduction in the effect induced by the application of a substance, due to an earlier exposure to that substance. For example,  administration of a toxin to rodents, at a given reference dose, induces a reproducible acute inflammatory response featuring a rise in a variety of immune system elements followed by a return to near-baseline conditions \cite{Beeson47,Cross02,Schade99,WestHeagy02}. If a small pre-conditioning dose of the toxin is given to an animal prior to the reference dose then  the activation of immune agents by the reference dose is attenuated.  This phenomenon is called tolerance.

A previous study \cite{Day06} analyzed tolerance in the context of  a four dimensional ordinary differential equation (ODE) model of the acute inflammatory response. Within the four dimensional ODE model, the origin represents
a healthy equilibrium state, and the abrupt administration of a toxin is represented by a jump of a trajectory to another point in phase space. Thus, starting from a given initial condition,  tolerance occurs precisely when the sequence of a pre-conditioning dose, a period of ensuing flow, and a subsequent reference dose leads to a trajectory position that is different from the one attained by direct administration of the reference dose, and from which a lower level of activated immune agents ensues.
From the observation of tolerance in the acute inflammatory response model, we reasoned that similar tolerance effects should be a general feature of trajectories generated from different initial conditions by a dynamical system with negative feedback.
Little analysis has been done on transient effects such as tolerance, compared to the major emphasis
in dynamical systems research on invariant manifolds and other structures derived from asymptotic and local calculations \cite{Guckenheimer83,Wiggins94}.

Our goal in this work is to provide a framework for the study of
tolerance in ODE systems.
Specifically, we focus on trajectories converging to an asymptotically stable node.
Overall, we are interested in necessary and sufficient conditions for tolerance, as we formally define it in Section 2. In a one-dimensional or scalar ODE, uniqueness of solutions prevents tolerance from occurring. Thus, we examine tolerance in two-dimensional ODE systems, using geometrical approaches. The general two-dimensional nonlinear case, which is treated in Section 3, poses challenges, since exact analytical solutions are generally not available.
However, through the use of isoclines and the concept of inhibition, we give some general results on conditions when tolerance can or cannot occur and we develop an approach to the derivation of more precise results for particular models.
Specific examples are used here to illustrate this approach. In Section 4,  we take advantage of analytical solutions to provide a complete analysis of tolerance in two-dimensional linear systems.
We finish with conclusions and a brief discussion of related work in Section 5.

\section{Preliminaries}
\subsection{Definitions and assumptions}
\label{Section: def}
In this section we present our assumptions and give the precise mathematical definition that we use for tolerance.
Consider the autonomous ODE system
\begin{equation}
\left\{
\begin{array}{ccc}
\dot{x} & = & f(x,y) \\
\dot{y} & = & g(x,y),
\end{array}
\right.  \label{system}
\end{equation}
where $x,y\in \mathbb{R}$, and $f, g$ are locally Lipschitz.\\

\begin{description}
\item[(A1)] Assume that there exists a stable fixed point of $($\ref{system}$)$, the eigenvalues of which are real and negative (to eliminate spirals and centers).  Without loss of generality, we will take $(0,0)$ as the given stable fixed point of (\ref{system}).\\
\label{A1}
\end{description}

 Let $\Gamma^+_{(0,0)}$ be the basin of attraction of $(0,0)$ in the first quadrant, $\mathbb{R}^{2+}\defeq [0,\infty)\times[0,\infty)$:
\begin{equation*}
\Gamma^+_{(0,0)} =\mathbb{R}^{2+} \cap \{(x,y)|(x,y)\cdot t \rightarrow (0,0) \text{ as } t \rightarrow \infty\},
\end{equation*}
where the notation $(x,y)\cdot t$ is the image of the point $(x,y)$ under the flow of $($\ref{system}$)$ for time $t$. The set of points, $\{(x,y)\cdot t|t\geq 0\}$, is the solution curve or trajectory of the initial value problem with initial value $(x,y)$. This set is also referred to as the graph of the solution.

 Let $\phi (t) = (\phi_1(t),\phi_2(t))$ and $\psi (t)=(\psi_1(t),\psi_2(t))$ be two solutions
 to the initial value problem of (\ref{system}) with initial values
\begin{equation}
\phi (0)=(x_{r},y_{r})\text{, } x_{r}>0\text{, }y_{r}\geq 0
\label{x0GreaterThanZero}
\end{equation}
and
\begin{equation}
\label{eq:psi0}
\psi (0)=( x_p, y_p),  x_p>0\text{, }y_p\geq 0.
\end{equation}

\begin{description}
\item[(A2)] Assume that both components of $\phi(t)$ and $\psi(t)$ are nonnegative for all $t\geq0$ and that $(x_r,y_r)$ and $(x_p,y_p)$ $\in \Gamma^+_{(0,0)}$.\\
\label{A2}
\end{description}

\begin{description}
\item[(A3)]
Assume that $x_r$ and $x_p$ are chosen such that $x_p \geq x_r$.\\
\label{A3}
\end{description}

\begin{definition}
Define $\phi (t)$ as the reference (R) trajectory or solution.\\ \label{Definition_OriginalSolution}
\end{definition}

\begin{definition}
Define $\psi (t)$ as the pre-conditioned or perturbed (P) trajectory or solution.\\ \label{Definition_CompetingSolution}
\end{definition}

Essentially, we are interested in determining whether or not there
exists a time when the first component of a P trajectory overtakes that of an
R trajectory, given that it was initially behind,  as they approach the
origin. Our ensuing discussion would apply equally if we considered
the second component instead of the first. \\

\begin{definition}
The system $($\ref{system}$)$ is said to \emph{exhibit tolerance for }$ \langle(x_r,y_r),(x_p,y_p)\rangle$ if there exists $\tau >0$ such that $\psi_1(\tau)<\phi_1(\tau )$.
\label{Definition 1}
\end{definition}\\

\begin{definition}
If $\psi_1(t)\geq \phi_1(t)$ for all $t\in \lbrack 0,\infty )$, then $(\ref{system})$ \emph{does not exhibit tolerance for }$\langle(x_r,y_r),(x_p,y_p)\rangle$.
\label{Definition 2}
\end{definition}\\

\begin{remark}
We will also use the terminology that \emph{$(x_p,y_p)$ or $\psi$ produces (or does not produce) tolerance in $($\ref{system}$)$} with respect to $(x_r,y_r)$ or $\phi$ to mean that Definition \ref{Definition 1} (Definition \ref{Definition 2}) holds. Figure \ref{FIG: TolDefinitions} illustrates definitions \ref{Definition 1} and \ref{Definition 2} with time courses of the first component of solutions $\phi(t)$ and $\psi(t)$ for a given $\langle(x_r,y_r),(x_p,y_p)\rangle.$\\
\end{remark}

\begin{figure}[h]
  \centering
  \includegraphics[width=5in]{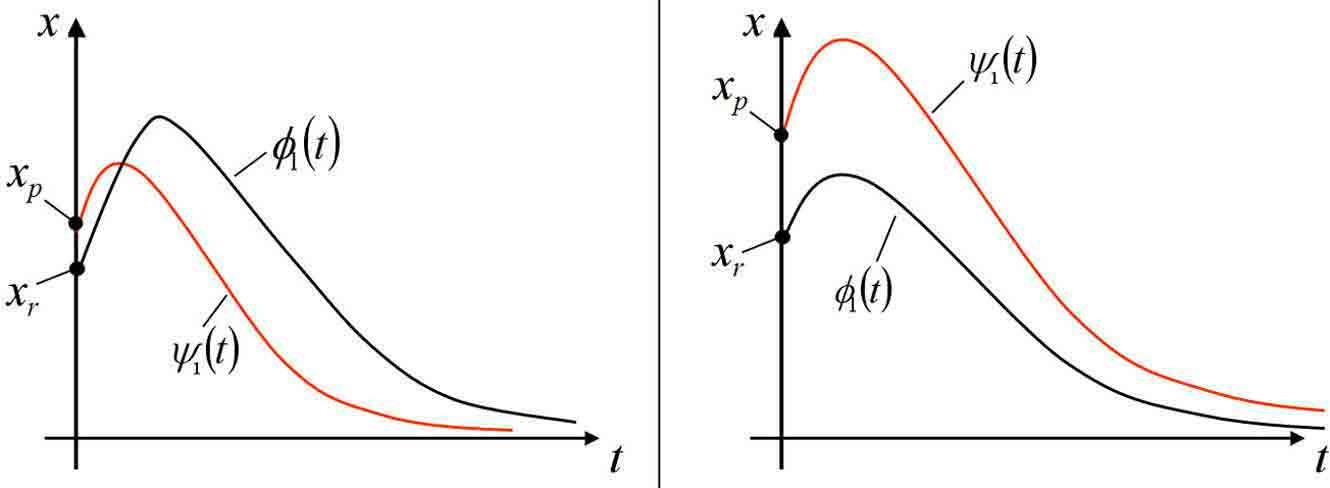}\\
  \caption{Illustration of Definitions \ref{Definition 1} and \ref{Definition 2}. Left (Right) panel:  Time course of the first component of a pre-conditioned (P) solution, $\psi(t)$, with initial condition $(x_p,y_p)$, which produces (does not produce) tolerance with respect to the reference (R) solution, $\phi(t)$, with initial condition $(x_r,y_r)$.}
  \label{FIG: TolDefinitions}
\end{figure}

\begin{remark}
Under (A3), $\psi(0) \defeq (x_p,y_p) \in [x_r,\infty) \times [0,\infty)$; that is, the initial value for the P solution
 could lie at any point on or to the right of the line $x=x_r$ in the first quadrant.
Correspondingly, we define $\Gamma^{x_r}_{(0,0)}$ to be the basin of attraction of $(0,0)$ in $[x_r,\infty)\times[0,\infty)\subset\mathbb{R}^{2+}$:
\begin{equation*}
\Gamma^{x_r}_{(0,0)} = \Gamma^+_{(0,0)} \cap [x_r,\infty)\times[0,\infty).
\end{equation*}

\end{remark}

\begin{remark}
The above definitions of tolerance are related to the biological setting that motivated this study through the interpretation of the P trajectory.
Consider a non-negative pre-conditioning solution $\rho(t)=(\rho_1(t),\rho_2(t))$  of (\ref{system}) with initial value
\begin{equation*}
\rho (0)=(x_\rho,y_\rho) \text{, } 0< x_{\rho}\leq x_r \text{, }0 \leq y_{\rho} \leq y_r.
\end{equation*}
 We then interpret the pre-conditioned solution $\psi (t)=(\psi_1(t),\psi_2(t))$ as the solution of (\ref{system}) with initial value
\begin{equation}
\label{eq:psi0_s}
\psi (0)=(x_p(s),y_p(s))\defeq \rho (s)+(x_h,y_h)\text{ for some }0\leq s<\infty,
\end{equation}
where $(x_h,y_h) \in \mathbf R^{2+}$.  If
$(x_h,y_h)=(x_r,y_r)$, which is typical for inflammation experiments,
 then for fixed $\phi(0)=(x_r,y_r)$ and $\rho(0)=(x_\rho,y_\rho)$, every $s$ defines a unique initial value for $\psi$ that satisfies (A3), namely $(x_p(s),y_p(s))$ as defined in
equation (\ref{eq:psi0_s}).
Thus, for a continuum  of $s$ values ranging from $0$ to $\infty $, a curve of possible $(x_p,y_p)$ values is formed, and it is of biological interest to know which of these $(x_p,y_p)$ lead to tolerance.
\end{remark}

\subsection{Properties of tolerance}
\label{Section: First Observations}

Definition \ref{Definition 1} refers only to the presence of tolerance at one time point $\tau >0$ such that $\psi_1(\tau)<\phi_1(\tau )$. However, continuity arguments can extend this window from a single time point to an open interval, $(t_1,t_2)$, around $\tau$, with $\psi_1(t_1)=\phi_1(t_1)$.
This observation is stated formally
 in Proposition \ref{Theorem 0} below and
 will be important in Section \ref{Tol2dNonLin}. Figure \ref{FIG: Theorem0} illustrates Proposition \ref{Theorem 0} with time courses of relevant solutions.\\

\begin{proposition}
\label{Theorem 0}Assume $($A1$)$, $($A2$)$, and $($A3$)$. If $($\ref{system}$)$ exhibits tolerance for $\langle(x_r,y_r),(x_p,y_p)\rangle$ at $\tau >0$, then there exists an open neighborhood $(t_1,t_2)$ around $\tau $ such that $\psi_1(\hat t)<\phi_1(\hat t)$ for every $\hat{t}\in (t_1,t_2)$ and $\psi_1(t_1)=\phi_1(t_1)$. Furthermore, $f(\psi (t_1))\leq f(\phi(t_1))$.
\end{proposition}

\begin{figure}[h]
\begin{center}
\includegraphics[width=3in]{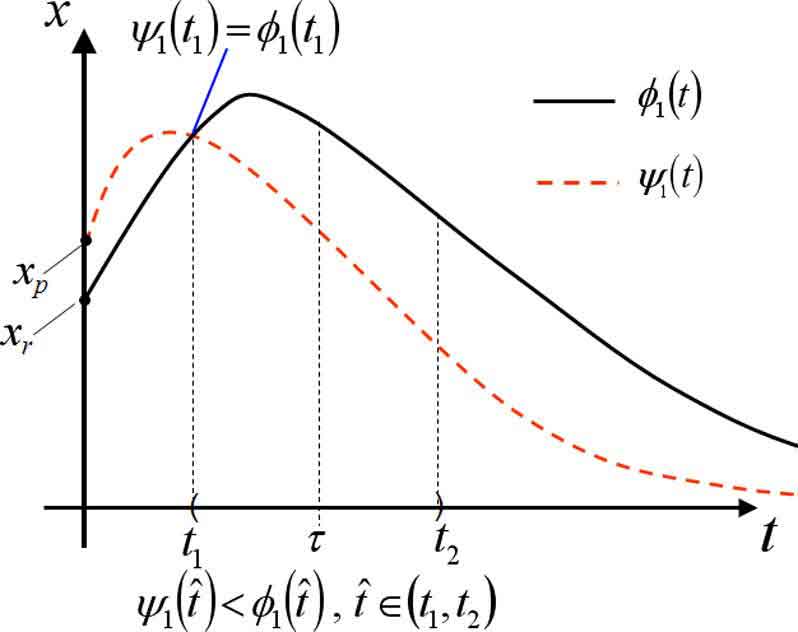}
\caption[Illustration for Proposition \protect\ref{Theorem 0}]{Time courses illustrating Proposition \protect\ref{Theorem 0}.  Note that in this example, $t_2$ could be chosen to be any $t>\tau$.}
\label{FIG: Theorem0}
\end{center}
\end{figure}
The window of tolerance can also be extended with respect to $(x_r,y_r)$ and $(x_p, y_p)$.\\

\begin{proposition}
Assume $($A1$)$, $($A2$)$, and $($A3$)$. If $($\ref{system}$)$ exhibits tolerance for $\langle(x_r,y_r),(x_p,y_p)\rangle$, then there exists an open ball, $B_r$, of radius $r$ around $(x_r,y_r)$ such that if $(x_k,y_k)\in B_r((x_r,y_r))\cap\Gamma^+_{(0,0)}$, then there exists a corresponding time $t_k>0$ such that tolerance is exhibited for $\langle(x_k,y_k),(x_p,y_p)\rangle$.
\label{THM: 1}
\end{proposition}\\

\begin{proposition}
Assume $($A1$)$, $($A2$)$, and $($A3$)$. If $($\ref{system}$)$ exhibits tolerance for given $\langle(x_r,y_r), (x_p,y_p)\rangle$, then there exists an open ball, $B_{\tilde r}$, of radius $\tilde r$ around $(x_p, y_p)$ such that if $(\tilde x_k,\tilde y_k)\in B_{\tilde r}((x_p,y_p))\cap\Gamma^{x_r}_{(0,0)}$, then there exists a corresponding time $\tilde t_k>0$ such that tolerance is exhibited for $\langle(x_r,y_r),(\tilde x_k,\tilde y_k)\rangle$.
\label{THM: 2}
\end{proposition}\\

Propositions \ref{THM: 1} and \ref{THM: 2} are easily proved by noting that solutions of (\ref{system}) are continuous and depend continuously on initial conditions. Each time $t_k$ or $\tilde t_k$ can also be extended to an interval of times for which tolerance occurs, by Proposition \ref{Theorem 0}.

\section{Conditions for the existence of tolerance}
\label{Tol2dNonLin}

In this section, we progressively build up a collection of ideas that are useful
for determining the set of initial conditions for P for which tolerance
can be guaranteed to occur or not to occur.
In particular, in subsection \ref{sub: basic}, we present a basic result on a general situation in which tolerance can be guaranteed to occur.
In subsection \ref{sub: isoclines}, we introduce some concepts that are useful
for refining the results from subsection \ref{sub: basic} and we discuss
their immediate consequences for tolerance.
We harness these ideas in subsection \ref{Subsection: TimeEstimates}, where we set up a general approach that can be used to move beyond the results from
subsections \ref{sub: basic} and \ref{sub: isoclines} in particular
systems, and we illustrate this approach in several examples in subsection \ref{sub: examples}.

\subsection{Basic conditions}
\label{sub: basic}
 In this subsection, we consider specific conditions on the initial values of $P$ and $R$ for
which tolerance can or cannot occur. We first consider conditions in which tolerance can occur when solutions $\phi (t)$ and $\psi (t)$ of system (\ref{system}), as defined in Section \ref{Section: def}, are subsets of the same solution curve.\\

\begin{proposition}
\label{Prop: Psi on Neg phi}Assume $($A1$)$, $($A2$)$, and $($A3$)$. Given $\langle(x_r,y_r), (x_p,y_p)\rangle$, assume $\phi _{1}(t)$
and $\psi _{1}(t)\rightarrow 0$ monotonically as $t\rightarrow\infty$. If there exists $\hat{t}>0$ such that $\phi(-\hat{t})=(x_p,y_p)$, 
$($\ref{system}$)$ does not exhibit tolerance for $\langle(x_r,y_r), (x_p,y_p)\rangle$.\\
\end{proposition}

This proposition follows immediately from the group property of flows and is the reason why tolerance is ruled out in one dimensional systems. Next, we focus on a situation where the reference trajectory $\phi$ is what we call an \textit{excitable} trajectory as represented, for example, in the left panel of Figure \ref{FIG: statement3.1}. We make this precise in terms of the graph of $\phi$, given by
\vspace{-.05in}
\begin{equation}
graph(\phi )=\left\{ (x,y) = (x_r,y_r)\cdot t : t\geq 0\right\},
\label{EQN: graph_phi}
\vspace{-.05in}
\end{equation}
with the following definition.
\vspace{.1in}
\begin{definition}
Assume that $($A1$)$, $($A2$)$, and $($A3$)$ hold.
Fix a positive integer $n$.
The trajectory $\phi(t)$ is $n$-excitable if there exist
times $t_{e_0}=0, t_{e_1}, \ldots, t_{e_{2n-1}}>0$ such that
\begin{description}
\item{$(a)$} $\phi_1(t_{e_i}) > x_r$ for all $i>0$,
\item{$(b)$} $g(\phi_1(t),\phi_2(t))>0$ for $t \in [0,t_{e_{2n-1}}]$,  and
\item{$(c)$}\vspace{-0.2in}
\[
\hspace{-.2in}\left\{ \begin{array}{l}
f(\phi_1(t),\phi_2(t))>0, \; t \in [t_{e_0},t_{e_1}) \textit{ and } (t_{e_{2i}},t_{e_{2i+1}}), \; i \in \{ 1, 2, \ldots, n-1 \}, \vspace{0.1in} \\
f(\phi_1(t),\phi_2(t))<0, \; t \in (t_{e_{2i+1}},t_{e_{2(i+1)}}), \; i \in \{ 0, 1, \ldots, n-2 \}, \mbox{ or} \; t > t_{e_{2n-1}}.
\end{array} \right.
\]\end{description}
\vspace{-.05in}
The trajectory $\phi(t)$ is excitable if it is $1$-excitable.
\label{def_excite}
\end{definition}\\

Excitable trajectories  are common in various biological
models. In the context of acute inflammation, an excitable trajectory represents the initial
activation of the immune system by a stimulus followed by a relaxation to
a stable baseline state.\\

\begin{remark}
Condition (b) on $g$ in Definition \ref{def_excite} is not necessary for our approach, but this assumption clarifies the presentation to follow.
\end{remark}\\

Below, we define a set $T$ such that tolerance with respect to $(x_r,y_r)$ occurs whenever $(x_p,y_p) \in T$, when $\phi(t)$ is an $n$-excitable trajectory.\\

\begin{definition}
For an $n$-excitable trajectory $\phi$, define $t_r>0$ to be the first positive time where $\phi_1(t_r)=x_r$, which exists since $\phi$ is $n$-excitable and continuous $($and by $(A1)$ and $(A2))$. Note also that $\phi_1(t)>\phi_1(t_r)=x_r$ for all $t\in(0,t_r)$ by definition of an $n$-excitable trajectory.
\end{definition}\\

\begin{definition}
Now, in terms of $t_r$, define $G$ to be the set of points $(x,y) \neq (x_r,y_r)$ on the graph of $\phi$ for $t\in (0,t_r]$:
\begin{equation}
G = \left\{ (x,y)|(x,y)=\phi(t) \text{ for } t \in (0,t_r]\right\}.
\end{equation}
\label{DEF: G}
\end{definition}

\begin{definition}
Assume that $\phi$ is an n-excitable trajectory. Define $L$ to be the line segment $L = \{x: x=x_r, y\in(y_r,\phi_2(t_r)]\}$ and define the region $S$ $($see Figure \ref{FIG: statement3.1}$)$ as the union of $L$ and the interior of the region bounded by $G$ and $L$.\\
\label{DEF: S}
\end{definition}

\begin{definition}
Define $T$ as the union of $G$ and $S$ as defined above,
\begin{equation}
T = G \cup S.
\end{equation}
\label{DEF: Sbar}
\end{definition}

\begin{definition}
Define $M=\max_{t\ge 0} \{\phi_1(t)\}$, which exists by $(A1)$, $(A2)$, and the continuity of $\phi$. Let $t_m>0$ $(t_M>0)$ be the minimal (maximal) positive time such that $\phi_1(t)=M$.\\
\end{definition}

\begin{proposition}
\label{Prop: Psi_inRegionS}
Let $\phi(0)=(x_r,y_r)$ and let $(x_p,y_p)$ be given.  Suppose that $($A1$)$, $($A2$)$, and $($A3$)$ hold and that $\phi$ is an $n$-excitable trajectory. Under these conditions, $T$ is a non-empty set. Moreover, if $(x_p,y_p)\in T$, then $($\ref{system}$)$ will exhibit tolerance for $\langle(x_r,y_r),(x_p,y_p)\rangle$.
\end{proposition}\\

\begin{proof}
By the assumptions, a region $T=G\cup S$ as defined above exists. We divide the proof into two parts since $T$ is defined as the union of two sets.

Part 1: Suppose $\psi (0)=(x_p,y_p)\in G$.  This implies that
$\psi (0) = (x_p,y_p) = \phi (\tau )$,
for some $\tau >0$. Again, $\phi_1(t)<M $ for all nonnegative $t>
t_M$. It follows that $\psi_1(t_M) = \phi_1(t_M+\tau) < M =
\phi_1(t_M)$.
Thus, $(\ref{system})$ exhibits tolerance for $\langle (x_r,y_r),(x_p,y_p)\rangle \in
G$ at time $t_M$.

Part 2:   Suppose $(x_p,y_p)\in S$. We first consider the case where $x_p>x_r$ and  define $t_p = \min_{t>0} \{t:\psi_1(t)=x_r\}$ such that $\psi(t)\in S$ for all $t\in [0,t_p]$.  If $t_p\ge t_r$ then since $t_r>t_M\ge t_m$,
$t_m\in (0,t_p)$.  Hence, $\psi_1(t_m)<M=\phi_1(t_m)$ and tolerance is exhibited at $t_m$.
Now, if  $0<t_p<t_r$, then it is possible that $\psi_1(t_m)>M$ (see bottom panel of Figure \ref{FIG: statement3.1}). However, from the definition of $t_r$, $\phi_1(t_p)>\phi_1(t_r) = x_r =\psi_1(t_p)$ and tolerance is exhibited at $t_p$.  Now, consider the special case that $x_p=x_r$.  If $f(x_p,y_p)>0$ then one of the above two cases holds.  If  $f(x_p,y_p)<0$, then there exists $\epsilon>0$ such that $\psi_1(\epsilon)<x_r$ and $\phi_1(\epsilon)>x_r$. Thus, $\phi(\epsilon)>\psi(\epsilon)$ and tolerance occurs at $\epsilon$.
\end{proof}\\

Figure \ref{FIG: statement3.1} illustrates Proposition \ref{Prop: Psi_inRegionS} in both phase space (left panel) and with time courses (right panel). Notice that if we consider the special case when $(x_r,y_r)$ of an $n$-excitable trajectory is on the $x$-axis, then uniqueness of solutions is sufficient to guarantee tolerance.
\begin{figure}[h]
\centering
\includegraphics[width=5in]{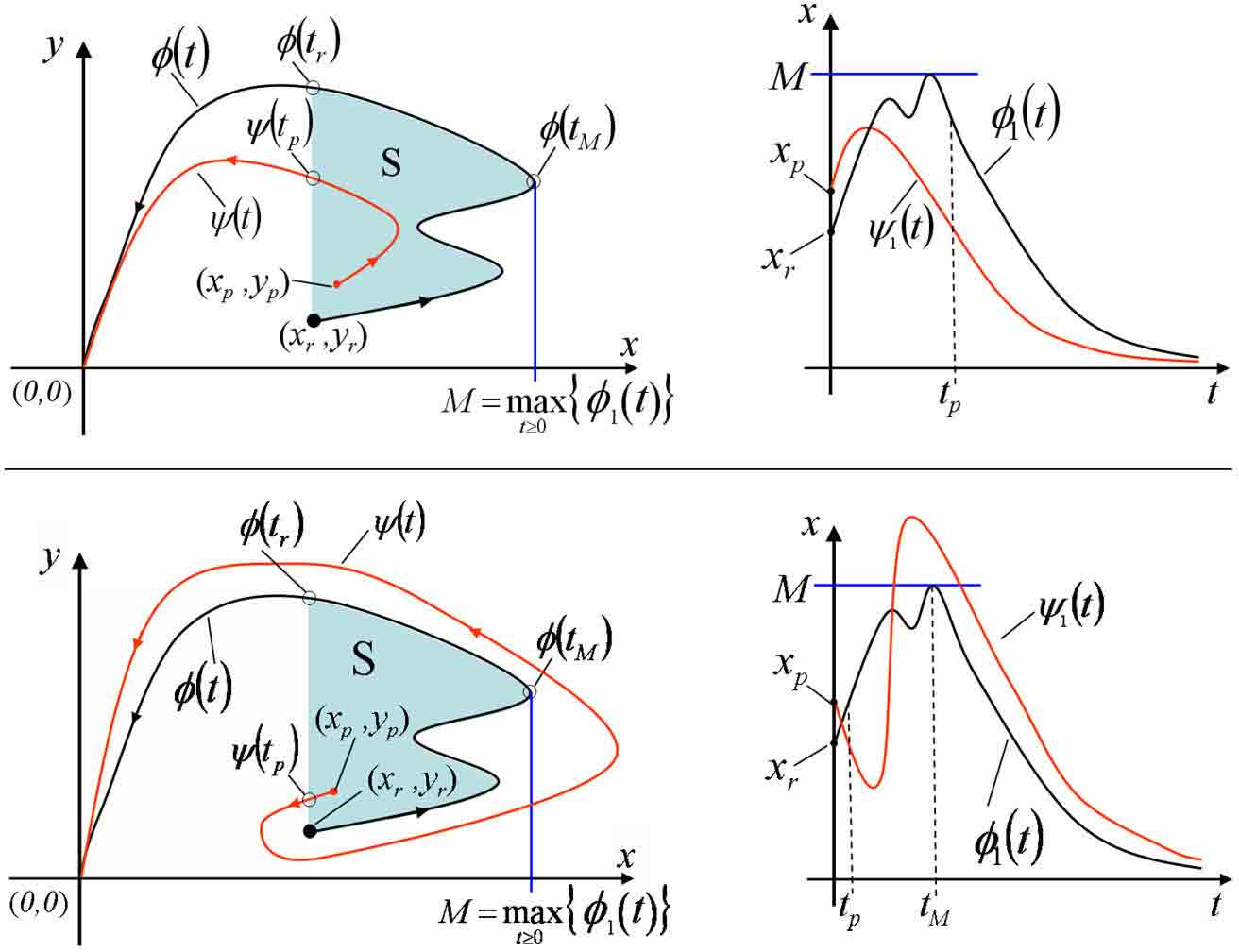}
\caption[Illustration for
Proposition \protect\ref{Prop: Psi_inRegionS}.]{Illustration of Proposition \protect\ref{Prop: Psi_inRegionS} in the case that $\phi$ is $n$-excitable. P trajectories with initial conditions in region $S$ exhibit tolerance. Left Panel: A $2$-excitable R trajectory, $\phi(t)$, initial condition, $(x_r,y_r)$ (black) and two example P trajectories, $\protect\psi(t)$, initial condition $(x_p,y_p)\in S$ (red). The maximum value in the $x$-direction for $\phi(t)$ is marked with a vertical blue line and denoted  by $M$.  Right Panel: Time courses of both $\phi_1(t)$ $($black$)$ and $\psi_1(t)$ $($red$)$.  Time $t_p$ is  where $\psi_1$ first takes on the value $x_r$ and $t_M$ is time when $\phi_1(t)$ last attains its maximal value.}
\label{FIG: statement3.1}
\end{figure}

If more constraints are imposed on the vector field $f$ then the
region that guarantees tolerance can be immediately expanded to include the strip above $T$ in $\Gamma^{x_r}_{(0,0)}$.  To be precise, we introduce the following definition.
\vspace{0.1in}
\begin{definition}
Define $\hat T$ by the set
\begin{equation}
\hat T = \left( (x_r, M) \times (\phi_2(t_M),\infty) \setminus T \right) \cap \Gamma^{x_r}_{(0,0)}.
\end{equation}\label{DEF: Shat}
\end{definition}
\vspace{-.1in}
\begin{proposition}
\label{THM: Inhibition_suf}  Assume $($A1$)$, $($A2$)$, $($A3$)$, and that $\phi$ is an $n$-excitable trajectory with $\phi(0)= (x_r,y_r)$. If  $f\le 0$ in $\hat T$, then for $(x_p,y_p)\in \hat T$, $($\ref{system}$)$ will exhibit tolerance.
\end{proposition}\\

\begin{proof}
For $(x_p,y_p)\in \hat T$ and $f\le 0$, it follows from the assumptions that $\psi_1(t) \le x_p<\phi_1(t_M)$ for $t\ge0$.  Thus, $\phi_1(t_M)> \psi_1(t_M)$. Hence, $(\ref{system})$ exhibits tolerance for $\langle(x_r,y_r),(x_p,y_p)\in \hat T>$ at time $t_M$.
\end{proof}\\

\subsection{Isoclines and Inhibition}
\label{sub: isoclines}
In the previous section we found generic conditions under which tolerance would
occur.  However, the initial conditions resulting in tolerance were confined to a
small region of the available basin of attraction.  Numerical experiments in various examples suggest that the region for tolerance is often larger.  Here, we introduce new concepts that enable us to expand the regions on which we can show that tolerance is possible or guaranteed.

\label{section: Isoclines}
Consider the ODE $($\ref{system}$)$ and assume $($A1$)$, $($A2$)$, and $($A3$)$ hold.\\

\begin{definition}
\ The $x$-\emph{isoclines} of $($\ref{system}$)$ are the family of curves
$($or level sets$)$, parametrized by a parameter $C \in \mathbb{R}$, each defined by $f(x,y)=C$.
\end{definition}\\

A nullcline, for instance, is an isocline for which $C=0$. The vector field points in the positive (negative) $x$-direction when $C$ is positive (negative).\\

\begin{remark}
We may define $y$-isoclines analogously to $x$-isoclines.
Since we do not consider these, we will drop the $x$- and just use \textit{isocline} to refer to the $x$-isoclines here. \\
\end{remark}

We now introduce the concept of \textit{inhibition}. Inhibition is a widely used term, especially in the context of mathematical models of biological systems, for the suppression of one quantity by another. However, the use of this term, while intuitive and heuristically understood, is not always mathematically precise. Hence, we give a precise definition of inhibition.
Subsequently, we prove two results relating to inhibition and tolerance.\\

\begin{definition}
\label{Definition_Inhibition} Given $\Omega \subseteq \mathbb{R}^{2+}$,
$y$ \emph{inhibits} $x$ in $\Omega$, and $\Omega$ is a \emph{region of inhibition} for (\ref{system}), if $f(x,y)$ is a monotone decreasing
function of $y$ in $\Omega$.\\
\end{definition}

\begin{remark}
Note that the sign of $f(x,y)$ is not specified in Definition \ref{Definition_Inhibition}.
Thus, when $y$ inhibits $x$, it may either slow the growth of $x$ or speed up its decay.
\end{remark}\\

A key first observation that follows from the definition of inhibition is that there is always the possibility of tolerance when $y$ inhibits $x$, as long as the perturbed trajectory samples larger $y$ values than the reference trajectory. We now formalize this observation by stating two further definitions and proving two preliminary results, which establish the necessity
of a region of inhibition and of certain relative positions of the perturbed and reference trajectories, respectively, for tolerance to exist. \\

\begin{definition}
The graph of $\psi $ is \emph{bounded below} by the graph of $\phi$ if $\phi_{2}(s_{1})<\psi_{2}(s_{2})$ whenever $\phi_{1}(s_{1})=\psi_{1}(s_{2})$ for any $s_{1},s_{2}>0$, not necessarily equal. For brevity, we say $\psi$ is \emph{bounded below} by $\phi$.\\
\end{definition}

\begin{proposition}
\label{THM: Inhibition_1} Assume that $($A1$)$, $($A2$)$, and $($A3$)$ hold and that
$\psi$ is bounded below by $\phi$. If $($\ref{system}$)$ exhibits tolerance for a given pair $\langle(x_r,y_r),(x_p,y_p)\rangle$, then there exist a region of inhibition $\Omega$ and $s_1,s_2\in \mathbb{R}^{+}$ such that $\psi_1(s_1)=\phi_1(s_2)$ with $\psi(s_1),\phi(s_2)\in \Omega.$
\end{proposition}\\

\begin{proof}
Assume that tolerance exists for $\langle(x_r,y_r),(x_p,y_p)\rangle$ but $y$ does
not inhibit $x$ in any region $\Omega$ that contains points $(\psi_1(s_1),\psi_2(s_1))$ and
$(\phi_1(s_2),\phi_2(s_2))$ where $\psi_1(s_1)=\phi_1(s_2)$ and $s_1,s_2\in \mathbb{R}^{+}$. Given tolerance, it follows from Proposition \ref{Theorem 0} that there exists $t^{\ast }$ such that
$\psi_1(t^{\ast})=\phi_1(t^{\ast})$ and $\psi_1(\hat{t})<\phi_1(\hat{t})$ for all $\hat{t}\in
(t^{\ast},t^{\ast}+\delta)$ for some $\delta >0$. Thus, $f(\psi(t^{\ast}))\leq f(\phi(t^{\ast}))$.  Since the graph of $\psi $ is bounded below by the graph of $\phi $, we have that at $t^{\ast }$, $\psi_2(t^{\ast})>\phi_{2}(t^{\ast })$. Our assumption that $y$ does not
inhibit $x$ in any region $\Omega$ containing the points $\psi(t^{\ast })$ and $\phi (t^{\ast })$ implies $f(\psi (t^{\ast}))>f(\phi (t^{\ast }))$, which is a contradiction. Hence, if $\psi $
is bounded below by $\phi$, and (\ref{system}) exhibits tolerance for $\langle(x_r,y_r),(x_p,y_p)\rangle$, there must exist a region of inhibition $\Omega$ and $s_1,s_2\in \mathbb{R}^{+}$, such that
$\psi_1(s_1)=\phi_1(s_2)$ and $\psi(s_1),\phi(s_2)\in \Omega.$
\end{proof}\\

\begin{remark}
Note that Propositions \ref{THM: Inhibition_suf} and \ref{THM: Inhibition_1} together imply that for an n-excitable trajectory to exist there must exist a region of inhibition.
\end{remark}\\

Proposition \ref{THM: Inhibition_1} states that a region of inhibition is necessary for tolerance to occur when the P trajectory, $\psi (t)$, is bounded \textit{below} by the R trajectory, $\phi(t)$.
However, for $\psi $ bounded \textit{above} by $\phi $, inhibition can be a detriment to the presence of tolerance under certain conditions.  First, we define what it means for $\psi$ to be \textit{bounded above} by $\phi $.\\

\begin{definition}
The graph of $\psi $ is \emph{bounded above} by the graph of $\phi$ if $\phi_2(s_1)>\psi_2(s_2)$ whenever $\phi_1(s_1)=\psi_1(s_2)$ for any $s_1,s_2>0$, not necessarily equal. For brevity, we say $\psi$ is \emph{bounded above} by $\phi$.\\
\end{definition}

\begin{proposition}
\label{THM: Inhibition_2} Assume that $($A1$)$, $($A2$)$, and $($A3$)$ hold. For $\langle(x_r,y_r),(x_p,y_p)\rangle$ such that $x_p >M$, if the graph of $\psi $ is bounded above by the graph of $\phi $, and $y$ inhibits $x$ in a region $\Omega$ such that $\phi (t)$, $\psi (t)\subset \Omega$ for all $t\geq 0$, then $($\ref{system}$)$ cannot exhibit tolerance for $\langle(x_r,y_r),(x_p,y_p)\rangle$.\\
\end{proposition}

\begin{proof}
The proof is analogous to that of Proposition \ref{THM: Inhibition_1}. If $x_p > \max_{t \geq 0} \phi(t)$, then tolerance requires $f(\psi(t^*)) < f(\phi(t^*))$ for some $t^*$ such that $\psi_1(t^*)=\phi_1(t^*)$, but this cannot occur in a region where $y$ inhibits $x$, given that $\psi$ is bounded above by $\phi$.
\end{proof}\\

Thus, Proposition \ref{THM: Inhibition_2} states that in order for tolerance to be a possibility for a P trajectory $\psi$ that is bounded above by the R trajectory $\phi $, for initial condition $\psi_1(0) >M$, it is necessary that there exists at least
one pair, $s_1,s_2\in \mathbb{R}^{+}$, such that $\psi_1(s_1)=\phi_1(s_2)$ and $\psi(s_1),\phi(s_2)$ do not belong to a region of inhibition.

Propositions \ref{Prop: Psi_inRegionS}, \ref{THM: Inhibition_suf}, \ref{THM: Inhibition_1}, and \ref{THM: Inhibition_2} suggest a strategy for evaluating whether or not tolerance may occur in a particular system for given R and P trajectories with initial values $\phi(0)=(x_r, y_r)$ and $\psi(0)=(x_p, y_p)$, under assumptions (A1), (A2), and (A3). First, if $\phi$ is an $n$-excitable trajectory, then by Proposition \ref{Prop: Psi_inRegionS}, tolerance occurs for all $(x_p,y_p)\in T$ (see Definition \ref{DEF: Sbar} and Figure \ref{FIG: statement3.1} ). If in addition $f\le 0$ in $\hat T$, then by Proposition \ref{THM: Inhibition_suf} tolerance occurs for all $(x_p,y_p)\in \hat T$ (see Definition \ref{DEF: Shat}).
Next, we identify the regions of inhibition for system (\ref{system}). If it can be established that the trajectory $\psi$ emanating from an initial condition $(x_p,y_p)$ is bounded below by $\phi$ but does not pass through a region of inhibition, then tolerance cannot occur (see Proposition \ref{THM: Inhibition_1}).  Similarly, if $x_p >M$, $\psi$ is bounded above by $\phi$, and $\psi, \phi$ are contained in a region of inhibition, then tolerance cannot occur (see Proposition \ref{THM: Inhibition_2}). If $f_y < 0$ on all of $\mathbb{R}^{2+}$, then the possibility of tolerance exists for all $(x_p, y_p)$ such that $\psi$ is bounded below by $\phi$.\\

\subsection{Time interval estimates}
\label{Subsection: TimeEstimates}
To obtain more precise conditions for the existence of tolerance, direct estimates regarding specific trajectories
of (\ref{system}) are necessary.  Here, we show how to derive estimates
for upper and lower bounds on the amount of time it takes for the relevant
trajectories to reach a specified $x$-value $x_f$ that is crossed by both
trajectories, $\phi(t)$ and $\psi(t)$.  If an $(x_p,y_p)$ can be found
such that $\psi(t)$ takes a
shorter time interval to reach $x_f$ than $\phi(t)$, then tolerance exists for
that $(x_p,y_p)$.

Assume that there is a positive integer $n$ for which the graph of $\phi$ can be decomposed into a union of $n$ graph segments such that the $y$ component of the graph is single valued with respect to
$x$ on each.  This assumption holds, for example, when $\phi$ is $m$-excitable
for some $m$.
Let $x_i$, $i\in\{1,\dots,n+1\}$ be the $n+1$ terminal points of the $n$ segments, defined by $x_1=x_r$,  $x_i  =  \phi_1(t^i_\phi)$, for $i=2,\dots,n$, where  $t^i_\phi=\inf_{t>t^{i-1}_\phi} \{t : f(\phi_1(t),\phi_2(t))=0\}$ with $t^1_\phi=0$, and $x_{n+1}=x_f$.  Let $t^{n+1}_\phi = \inf_{t>t^n_\phi}\{ t : \phi_1(t)=x_f\}$. The total time to traverse the trajectory from $x_r$ to $x_f$ is then given by $t_\phi=\sum_{i=1}^{n}\Delta t^i_\phi$, where $\Delta t^i_\phi= t^{i+1}_\phi-t^i_\phi$.

On each graph segment we can express the graph of $\phi$ as a function $y=v_i(x)$, where $v_i$ is
defined on the interval $x_i\le x\le x_{i+1}$, $i\in\{1,\dots,n\}$. We can compute $\Delta t_i$ for each segment directly by integrating the first equation of (2.1) along the graph segment defined by $y=v_i(x)$, i.e. $\dot{x} = f(x,v_i(x))$, to obtain
\begin{equation}
t_\phi= \sum_{i=1}^{n} \int_{x_i}^{x_{i+1}} \frac{du}{f(u,v_i(u))}.
\label{tphi1}
\end{equation}
A similar construction can give $t_\psi$, with initial $x$-coordinate
$x_p$.  Tolerance then implies  $t_{\psi }< t_{\phi}$.  In general, it is not possible to obtain $v_i$ in closed form, but depending on the structure of $f$, estimates can be made to obtain various bounds for $t_\phi$ and $t_\psi$.

For example,  with respect to (\ref{system}), consider the family of $x$-isoclines $f(x,y)=C$, where $C\in \mathbb{R}$.  Let $c_{\phi}^i=\sup_{t\in [t^i_\phi,t^{i+1}_\phi)}\{|f(\phi_1(t),\phi _2(t))|\}$, i.e. the largest magnitude isocline through which the trajectory  $\phi$ passes on the segment $[x_\phi^i,x_\phi^{i+1}]$.  Then from (\ref{tphi1}) we obtain $ t_\phi\ge \sum_{i=1}^{n} |x_\phi^{i+1}-x_\phi^{i}|/c_\phi^i$.  Likewise, let $c_{\psi }^i=\inf_{t\in \lbrack t^i_\psi,t^{i+1}_\psi)}\{|f(\psi _{1}(t),\psi _{2}(t))|\}$, i.e. the smallest magnitude isocline through which the trajectory $\psi (t)$ passes on the segment $[x_\psi^i,x_\psi^{i+1}]$, yielding $ t_\psi\le \sum_{i=1}^{n} |x_\psi^{i+1}-x_\psi^{i}|/c_\psi^i$. Thus, if
 \begin{equation}
\sum_{i=1}^{n} \frac{|x_\psi^{i+1}-x_\psi^{i}|}{c_\psi^i}< \sum_{i=1}^{n} \frac{|x_\phi^{i+1}-x_\phi^{i}|}{c_\phi^i},
 \label{timecond}
 \end{equation}
then $t_{\psi }<t_{\phi}$,
which implies tolerance.

We can use condition (\ref{timecond}) to show, for example, that
if $\psi(t)$ is bounded below by an $m$-excitable trajectory $\phi(t)$, and $\phi(t)$ and $\psi(t)$ both lie in a region of inhibition, then the region on which tolerance is guaranteed to occur can be expanded from that defined in Proposition \ref{THM: Inhibition_suf}.
As an example, suppose that $\phi(t)$ is an excitable trajectory.  We can then divide $\phi$ into two segments.  In the first segment $\phi_1(t)$ and $\phi_2(t)$ are increasing, and in the second $\phi_1(t)$ is decreasing. By continuity and (A1),  $\phi_2(t)$ must first increase and then decrease on the second segment.  The end point of the first segment is $x_M=\max_{t>0} \phi_1(t)$.
Define $x_f$ as the $x$-value where $\phi_2(t)$ is maximal and let $\phi_2(t)=y_f$ at this point. Since $\phi(t)$ belongs to a region of inhibition, the largest magnitude isocline through which the first segment of $\phi(t)$ passes is given by $c_\phi^1=f(x_r,y_r)=C_r$.  On the second segment, the largest magnitude isocline passes through $\phi(t)$ when $\phi_2(t)$ is maximal.  Thus $c_\phi^2=|f(x_f, y_f)|=C_f>0$.

\begin{figure}[h]
\begin{center}
\includegraphics[width=3in]{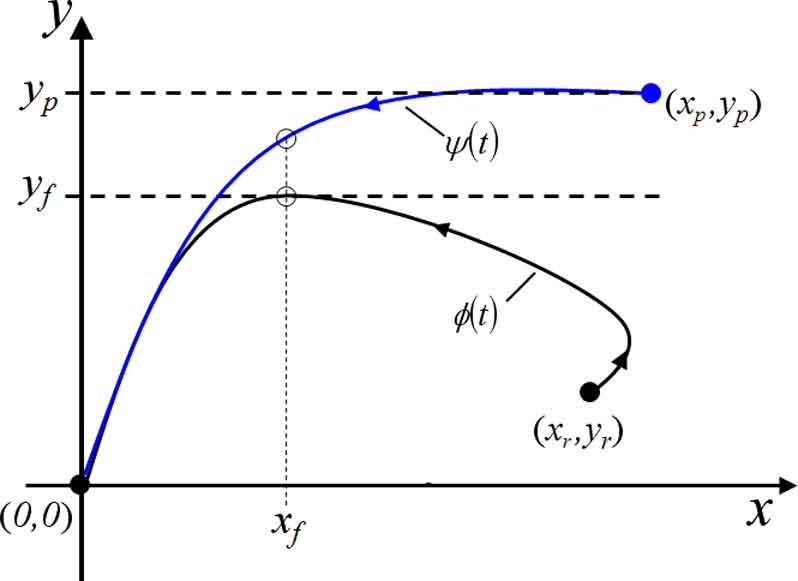}
\caption[Illustration for time interval estimates]{Illustration for time interval estimates.}
\label{FIG: TimeEstimates}
\end{center}
\end{figure}

Now, using Figure \ref{FIG: TimeEstimates} as a reference, consider a trajectory $\psi(t)$ such that $f<0$ along the trajectory, so there is only one segment and it is bounded below by the line $y=y_f$.  Thus, $c_\psi^1=C_\psi > C_f$, and tolerance is observed if
\begin{equation}
\frac{|x_f-x_p|}{C_\psi}<\frac{|x_M-x_r|}{C_r} + \frac{|x_f-x_M|}{C_f}.
\label{timeregion}
\end{equation}
If we consider an excitable trajectory, then $x_p>x_f$, $x_M>x_r$, and $x_M>x_f$.  Taking these inequalities in (\ref{timeregion}) gives the tolerance condition
\begin{equation}
x_p<x_M+\frac{C_\psi-C_f}{C_f}(x_M-x_f)+\frac{C_\psi}{C_r}(x_M-x_r)\defeq \hat{x}_M.
\label{regionexpand}
\end{equation}
Since $C_\psi>C_f$, (\ref{regionexpand}) implies that  $\hat{x}_M>x_M$, which expands the region obtained from Proposition \ref{THM: Inhibition_suf}.  We note that $C_\psi$ is a function of $y_p$, so (\ref{regionexpand}) defines a region $R$ such that if $(x_p,y_p)\in R$, then tolerance occurs in (\ref{system}).\\

\subsection{Examples}
\label{sub: examples}
In the examples below, we illustrate the ideas introduced in the previous subsection. \\

\begin{example}
\label{EX: Isoclines_NL_Example2} Consider the system given by
\begin{equation}
\left.
\begin{array}{ccccc}
\dot{x} & = & f(x,y) & = & \frac{x^{2}}{1+y}-x \\
\dot{y} & = & g(x,y) & = & x^2-\frac{y}{2}
\end{array}%
\right\} \text{.}
\label{System: NLExample2}
\end{equation}
\end{example}\\
Note that $(0,0)$ is a stable node for (\ref{System: NLExample2}).
The isoclines for this system are the family of curves given by the equation
\begin{equation}
y=\frac{x^{2}-x-C}{x+C}  \label{EQN: Iso_NL_ex2}
\end{equation}
for $C\in \mathbb{R}$. Figure \ref{FIG: Isoclines_Ex2} shows a subset of the isoclines for $C\in [-4.0,50]$ shown in increments of $0.5$ for those above the $C=0$ isocline and in increments of $1.0$ for those below.

\begin{figure}[h]
\begin{center}
\includegraphics[width=4in]{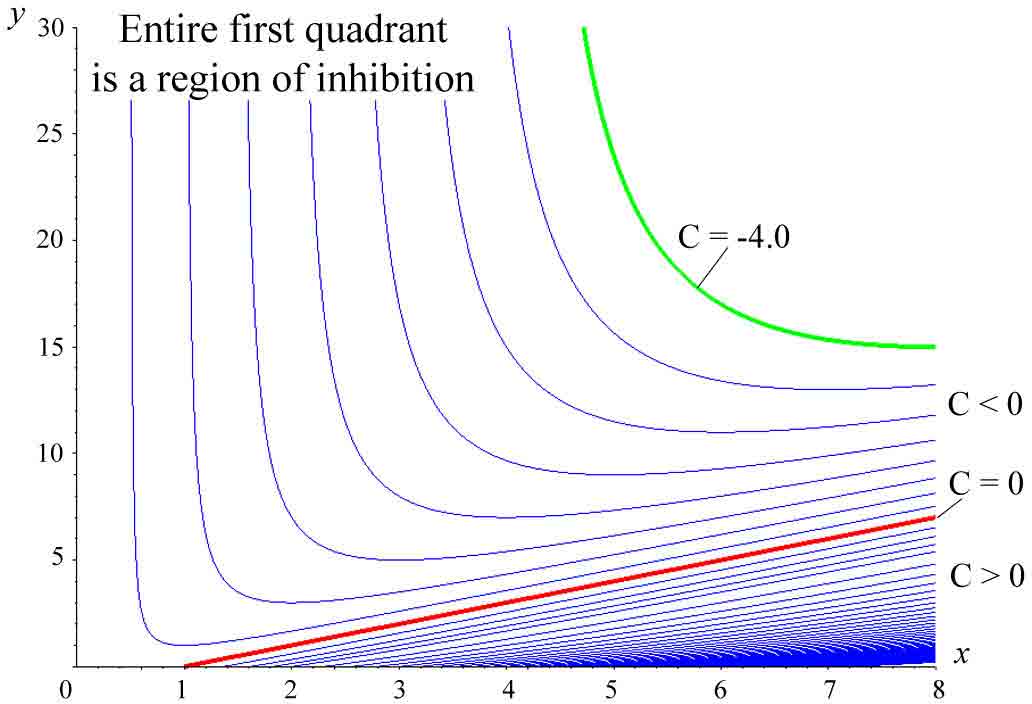}
\caption[Isoclines for Example \ref{EX: Isoclines_NL_Example2}]{Isoclines for Example \ref{EX: Isoclines_NL_Example2} defined by Equations \ref{EQN: Iso_NL_ex2} for $C\in \lbrack -4.0,50]$.}
\label{FIG: Isoclines_Ex2}
\end{center}
\end{figure}

For each $C<0$, the corresponding isocline has a local minimum at $x=-2C$ and a vertical asymptote at $x=-C$. Direct differentiation of $f$ in (\ref{System: NLExample2}) yields $f_y <0$, or equivalently, from (\ref{EQN: Iso_NL_ex2}), $dy/dC<0$, for all $(x,y)$ in the first quadrant. Thus, the entire first quadrant is a region of inhibition. We will consider several different initial conditions $(x_r,y_r)$ for $\phi(t)$ in this example:\\
\renewcommand{\labelenumi}{(\alph{enumi})}
\begin{enumerate}
\item $(x_r,y_r)=(4.0, 0.0)$,
\item $(x_r,y_r)=(4.0, 3.0)$,
\item $(x_r,y_r)=(4.0,10.0)$.\\
\end{enumerate}

\begin{figure}[h]
\begin{center}
\includegraphics[width=2.838in]{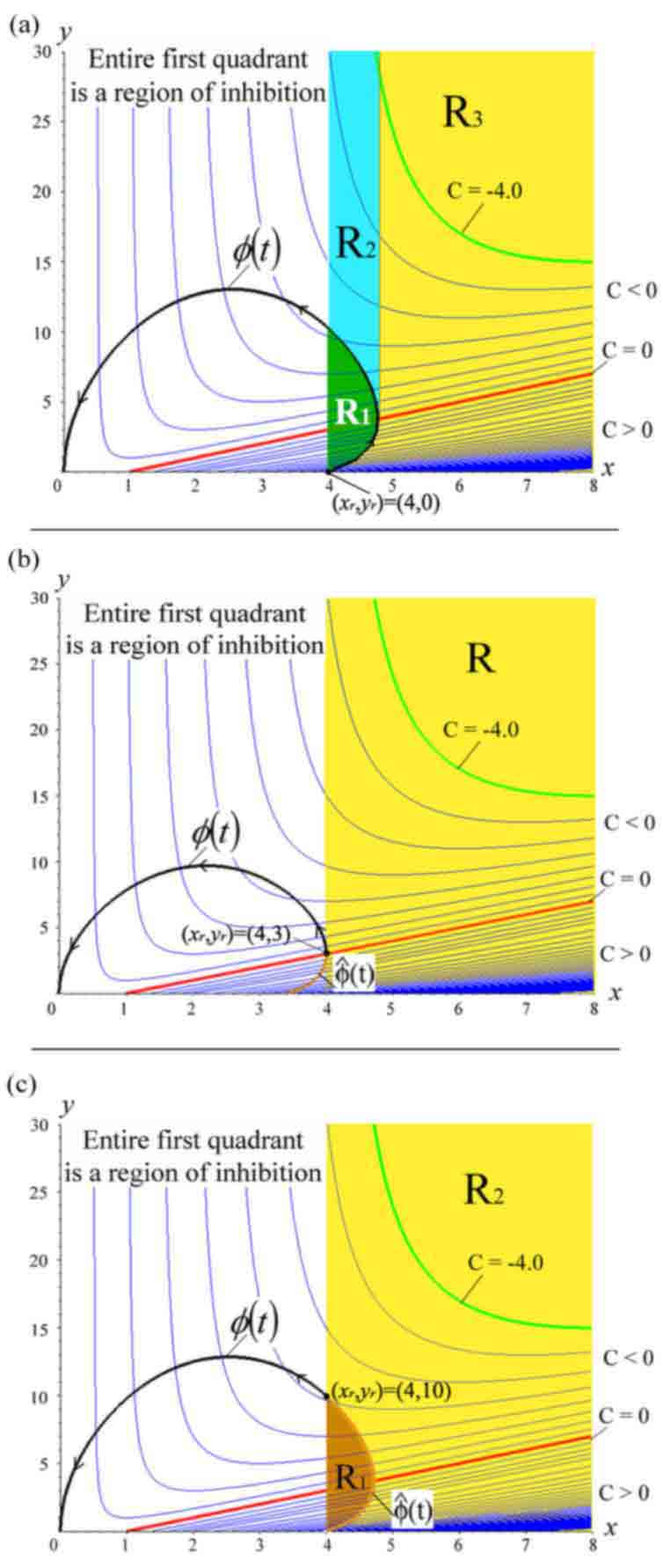}
\caption[Isoclines and $(x_r,y_r)$ points for Example \ref{EX: Isoclines_NL_Example2}. ]{Isoclines and various initial values $(x_r,y_r)$ for Example \ref{EX: Isoclines_NL_Example2} a, b, and c.}
\label{FIG: Isoclines_Regions_Ex2}
\end{center}
\end{figure}

\noindent For initial condition (a), Figure \ref{FIG: Isoclines_Regions_Ex2}(a) displays the following features:
\begin{itemize}
\item $\phi(t)$ is the curve shown in black for initial condition $\phi(0)=(x_r,y_r)=(4.0,0)$.

\item $R_1$ is the green region union the boundary of $\phi(t)$ and is defined in the same manner as the region $T$, in Definition \ref{DEF: Sbar}.
$R_1$ is bounded to the left by $\{ x=x_r \}$, in accordance with $(A3)$, and to the right by $\{ x=x_M \}$.

\item $R_2$ (the light blue region) is the strip in $\Gamma^{x_r}_{(0,0)}$, lying above $R_1$, sharing its bounds on $x$.

\item $R_3$ (the yellow region) is the complement of $R_1\cup R_2$ with respect to $\Gamma^{x_r}_{(0,0)}$, namely $R_3 \defeq \Gamma^{x_r}_{(0,0)}\backslash (R_1 \cup R_2).$\\
\end{itemize}

Case 1(a):  $(x_p, y_p) \in R_1$.  By Proposition \ref{Prop: Psi_inRegionS}, any $(x_p, y_p) \in R_1$ will produce tolerance.  Furthermore, define $G=\left\{ (x,y)|(x,y)=graph(\phi)\cap [x_r,\infty)\times (0,\infty)\right\}$.  By Proposition \ref{THM: 2}, for each $(x_p, y_p)\in G$, there exists an open ball, $B_{\tilde r}$, of radius $\tilde r$ around $(x_p, y_p)$ such that $(\tilde x_k,\tilde y_k) \in B_{\tilde r} \cap \Gamma^{x_r}_{(0,0)}$ produces tolerance with respect to $(4.0,0)$.

\pagebreak
Case 2(a):  $(x_p, y_p) \in R_2$.  Region $R_2$ is a region of inhibition in which $f<0$. Thus, by Proposition \ref{THM: Inhibition_suf}, any $(x_p, y_p) \in R_2$ will produce tolerance.\\

Case 3(a): $(x_p, y_p) \in R_3$.  In this case, for $\psi_1(t)> M \defeq \max_{t\geq 0}\{\phi_1(t)\}$, $\psi(t)$ is bounded below by $\phi(t)$ and the presence of a region of inhibition makes tolerance possible (Proposition \ref{THM: 2}). For $\psi_1(t)< M$, which is possible for small $y_p$, $\psi(t)$ will eventually be bounded below by $\phi(t)$ and hence tolerance is again possible.\\

Figure \ref{FIG: ScreenShotEX317a} contains links to four separate animations that illustrate the presence or absence of tolerance in Example \ref{EX: Isoclines_NL_Example2}(a) using various choices of $\psi(0)$ from the different regions shown in Figure \ref{FIG: Isoclines_Regions_Ex2}(a). Each animation displays both phase space trajectories of $\phi$ and $\psi$ and time courses of $\phi_1(t)$ and $\psi_1(t)$ in a side-by-side comparison.\\

\begin{figure}
  \includegraphics[width=5in]{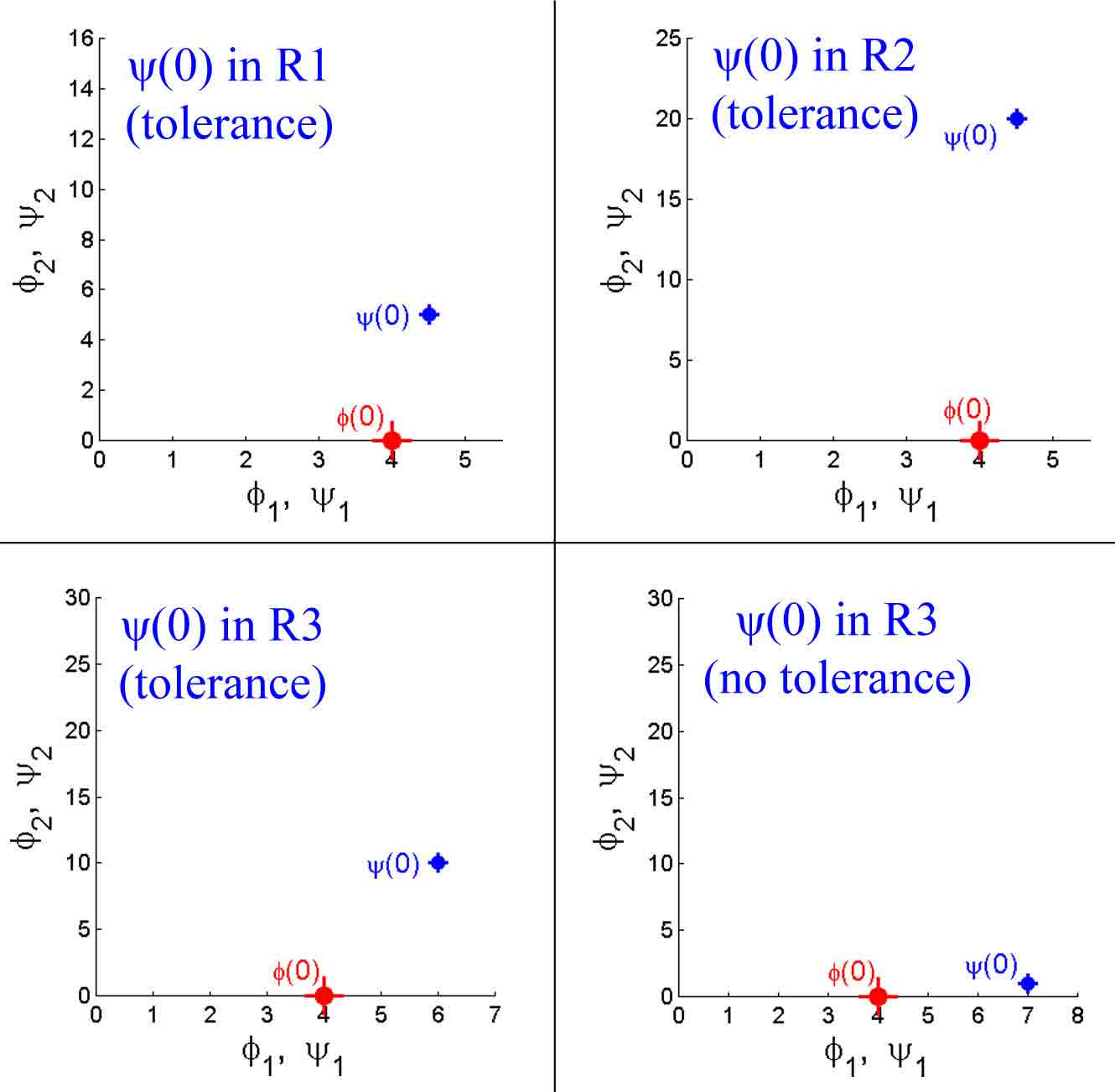}\\
  \caption{Animations for Example \ref{EX: Isoclines_NL_Example2}$(a)$, showing the presence or absence of tolerance with respect to $\phi(0)=(4,0)$ for differing choices of $\psi(0)$. $\phi(0)$ is denoted by the large red dot and $\psi(0)$ is denoted by the smaller blue dot.  Given $\phi(0)=(4.0,0)$, $\psi(0)=(4.5, 5)\in R_1$ produces tolerance $($Top Left$)$, $\psi(0)=(4.5, 20) \in R_2$ produces tolerance $($Top Right$)$, $\psi(0)=(6,10)\in R_3$ produces tolerance $($Bottom Left$)$, and $\psi(0)=(7,1) \in R_3$ does not produce tolerance $($Bottom Right$)$. }\label{FIG: ScreenShotEX317a}
\end{figure}

\newpage
If $y_r$ is increased with $x_r$ fixed, the regions $R_1$ and $R_2$ shrink.
Finally, when $y_r$ reaches 3.0, corresponding to initial condition (b), these regions disappear. Figure \ref{FIG: Isoclines_Regions_Ex2}(b) displays the following features:
\begin{itemize}
\item $\phi(t)$ is the curve shown in black for initial condition $\phi(0)=(x_r,y_r)=(4.0,3.0)$.

\item The orange curve, denoted as $\hat{\phi}$, is the curve of points obtained by integrating $\phi(t)$ backwards in  time from $t=0$ to $t\approx -1.0$, at which time it intersects the $x$-axis at $\hat{x}\approx 3.4$.

\item $R$ is the yellow region defined to be $\Gamma^{x_r}_{(0,0)} \setminus (4.0,3.0)$.

\end{itemize}
For this example, if $x_p=x_r=4.0$, then for all $(x_p,y_p)\in R$, the corresponding graph of $\psi$ is or will eventually be bounded below by the graph of $\phi$.
Since the graph of $\psi$ lies in $\mathbb{R}^{2+}$ and $\mathbb{R}^{2+}$ is a region of inhibition, Proposition \ref{THM: Inhibition_1} implies that it is possible that tolerance can be exhibited by any $(x_p,y_p)\in R$, although, as in the previous case, tolerance is not guaranteed (see Figure \ref{FIG: Isoclines_Regions_Ex2}(b)).\\

For $y_r > 3.0$, the situation is qualitatively similar to that shown in Figure \ref{FIG: Isoclines_Regions_Ex2}(c) for initial condition (c), $(x_r,y_r)=(4.0,10.0)$.  Figure \ref{FIG: Isoclines_Regions_Ex2}(c) displays the following features:
\begin{itemize}
\item $\phi(t)$ is the curve shown in black for initial condition $\phi(0)=(x_r,y_r)=(4.0,10.0)$.

\item The orange curve, denoted as $\hat{\phi}$, is the curve of points obtained by integrating $\phi(t)$ backwards in time from $t=0$ to $t\approx -.58$, at which time it intersects the $x$-axis at $\hat{x}\approx 4.0$.

\item $R_1$ is the orange region union its boundaries: (1) $\hat \phi(t)$ and (2) the line segment $\{(x,y)|x=x_r,y\in[0,10]\}$.

\item $R_2$ is the yellow region defined to be the complement of $R_1$ with respect to $\Gamma^{x_r}_{(0,0)}$, namely $R_2 = \Gamma^{x_r}_{(0,0)} \setminus R_1.$\\
\end{itemize}

Case 1(c):  $(x_p, y_p) \in R_1$.  Using Proposition \ref{THM: Inhibition_2} and Proposition \ref{Prop: Psi on Neg phi}, $(x_p, y_p) \in R_1$ cannot produce tolerance with respect to $(x_r,y_r)=(4.0,10.0)$.\\

Case 2(c):  $(x_p, y_p) \in R_2$.  For all $(x_p,y_p)\in R_2$, the corresponding graph of $\psi$ is or will eventually be bounded below by the graph of $\phi$. Again, tolerance is possible but not guaranteed.\\

In summary of initial condition (c), given that the entire first quadrant is a region of inhibition, there is the possibility of tolerance for all $(x_r,y_r)$ and $(x_p, y_p)$ except when $x_r \leq x_p < \max_{t \geq 0} \hat \phi_1(t)$ and $\psi$ is bounded above by $\phi$, as illustrated in the orange region $R_1$ in Figure \ref{FIG: Isoclines_Regions_Ex2}(c). Figure \ref{FIG: ScreenShotEX317c} links to three animations for Example \ref{EX: Isoclines_NL_Example2}(c) with $\psi(0)$ chosen from the different regions shown in Figure \ref{FIG: Isoclines_Regions_Ex2}(c). As before, each animation shows phase space and time courses in a side-by-side comparison.\\

\begin{figure}
  \includegraphics[width=5in]{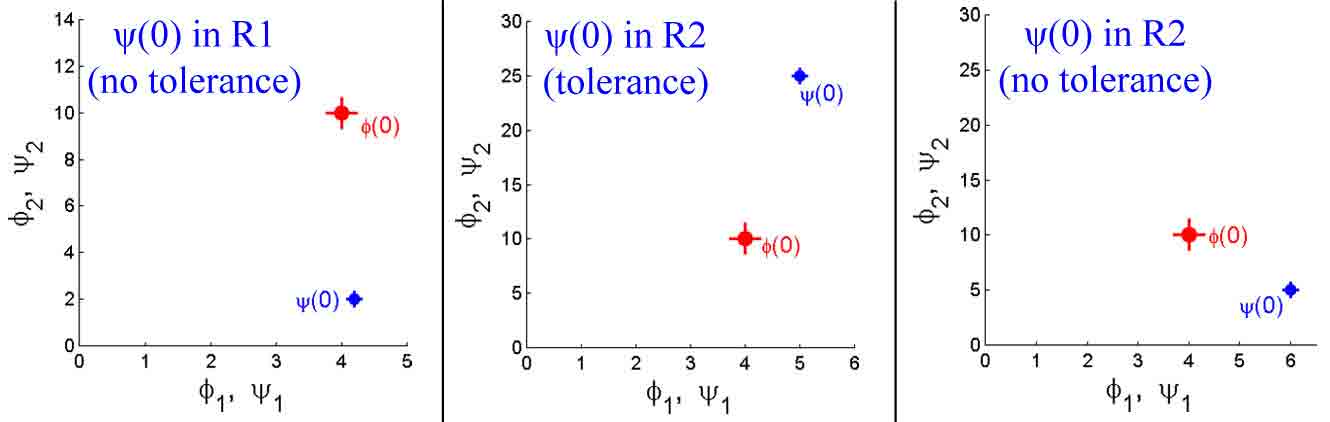}\\
  \caption{Animations for Example \ref{EX: Isoclines_NL_Example2}$(c)$, showing the presence or absence of tolerance with respect to $\phi(0)=(4,10)$ for differing choices of $\psi(0)$. $\phi(0)$ is denoted by the large red dot and $\psi(0)$ is denoted by the smaller blue dot.  Given $\phi(0)=(4,10)$, $\psi(0)=(4.2, 2)\in R_1$ does not produce tolerance $($Left$)$, $\psi(0)=(5, 25) \in R_2$ produces tolerance $($Middle$)$, and $\psi(0)=(6,5)\in R_2$ does not produce tolerance $($Right$)$.}\label{FIG: ScreenShotEX317c}
\end{figure}

We now use time interval estimates to expand the region that guarantees tolerance.  Consider initial value (a).
We choose $(x_f,y_f)$ such that $y_f = \max_{t \geq 0}\phi_2(t)$.  We note that the extremal points of  $\phi(t)$, $(x_M,y_M)$ and  $(x_f,y_f)$,  are on the $x$-nullcline and $y$-nullcline respectively so that $y_M=x_M-1$ and $y_f=2x_f^2$.  Given that initial value (a) results in an excitable trajectory,
we can apply (\ref{regionexpand}) with $C_r=12$ and $C_\psi=C_f=|x_f^2/(1+2 x_f^2)-x_f|$.  This then establishes a bound on $\hat{x}_M$, such that tolerance
occurs for $x_r < x_p < \hat{x}_M$, in terms of the initial value and extremal points  of the reference trajectory $\phi(t)$.   For example, rough bounds on $x_f$ and $x_M$ can be obtained from a visual inspection of $\phi(t)$.  From Fig \ref{FIG: Isoclines_Regions_Ex2},  we can propose $2<x_f<3$, leading to $1.55<C_f<2.53$, and $4.5<x_M<5$, with $\hat{x}_M = x_M + (C_f/C_r)(x_M-x_r)$ from (\ref{regionexpand}) with $C_\psi=C_f$. More stringent bounds can be obtained by performing numerical integration using interval arithmetic.  Moreover, as $y_p$ increases, $C_\psi$ increases while $C_f$ remains fixed, such that tolerance can be guaranteed for larger $x_p$, given larger $y_p$.

In fact, example \ref{EX: Isoclines_NL_Example2} is simple enough that we can obtain more precise estimates on $t_\phi$ and $t_\psi$, as defined in Section \ref{Subsection: TimeEstimates}. Let $t_{\phi}$ (similarly, $t_{\psi}$) be the time of passage from $\phi_1=x_r$ ($\psi_1=x_p$) to $\phi_1=x_f$ ($\psi_1 = x_f$). $\phi(t)$ can be represented by two segments. Denote the graph of $\phi$ for $t \in [0,t_{\phi}]$ by $(u,v_i(u))$, $i=1,2$ on the two segments.    $t_\phi$  is given by (\ref{tphi1}), with $x_\phi^1=x_r$, $x_\phi^2=x_M$ and $x_\phi^3=x_f$, where $x_M=\max_{t>0}\phi_1(t)$.  Recall that in this example, the entire first quadrant is a region of inhibition.  Our approach is to estimate the time intervals by setting $v_i(u)$ to a constant in (\ref{tphi1}) and then integrating to obtain $t_\phi>\Delta(y_r,x_r,x_M)+\Delta(y_f,x_M,x_f)$, where
\begin{equation}
\Delta(w,a,b)=\int_{a}^{b}\frac{du}{u^2/(1+w)-u}= \log\frac{|1+w-b|}{|1+w-a|} +
\log\frac{a}{b}.
\label{tw}
\end{equation}

Next, we compute $t_\psi$ for the trajectory $\psi(t)$ with initial condition
$(x_p,y_p)$ and ending at $(x_f,y_f)$.
Now, consider those $(x_p,y_p)$ such that $x_p > x_r$ and
$y_p > y_f$.
Since the $y$-nullcline is the curve $y=2x^2$, by uniqueness of solutions to (\ref{System: NLExample2}), the latter condition ensures
that $\psi_2(t) > y_f$ for all $t$ such that $ \psi_1(t)>x_f$.
By the continuity of $\Delta(w,x_f,x_p)$ in $w$,
$t_\psi=\Delta(y_\psi,x_p,x_f)$ for some $y_\psi>y_f$.  Thus, for the tolerance condition $t_\psi < t_\phi$ to hold, it is sufficient that
\begin{equation}
\Delta(y_r,x_r,x_M)+\Delta(y_f,x_M,x_f)>\Delta(y_\psi,x_p,x_f).
\label{delcond}
\end{equation}
If $x_p = x_M$, then the observation that $\Delta(y_f,x_M,x_f) > \Delta(y_\psi,x_M,x_f)$ implies that (\ref{delcond}) holds, and hence tolerance
occurs, as expected from Proposition \ref{THM: Inhibition_suf}.
For $x_p > x_M$, writing
 $\Delta(y_\psi,x_p,x_f)=\Delta(y_\psi,x_p,x_M)+\Delta(y_\psi,x_M,x_f)$ shows
immediately that the upper bound for tolerance can be extended from
$x_M$ to some $x_p > x_M$.

Assuming that both sides are positive, as in Figure \ref{FIG: Isoclines_Regions_Ex2}, condition (\ref{delcond}) can be expressed as
\begin{equation}
\frac{x_r(1-x_f+y_f)(x_M-1-y_r)}{(1-x_M+y_f)(x_r-1-y_r)}>\frac{(1+y_\psi-x_f)x_p}{1+y_\psi-x_p}.
\label{delcond2}
\end{equation}
Condition (\ref{delcond2}) still depends on $y_\psi$, which
can be estimated under the assumption that $y_\psi\ge \psi_2(t_\psi)$ (which holds, for example, if $g<0$ along $\psi(t)$ from $t=0$ to $t=t_\psi$).  Formally integrating the second equation of  (\ref{System: NLExample2}) gives $\psi_2(t_\psi)=y_pe^{-t_\psi/2} +\int_0^{t_\psi} e^{-(t_\psi-t')/2} x^2 dt'$.
On the trajectory $\psi(t)$, $x_f\le x\le x_p$, hence $\psi_2(t_\psi)>y_pe^{-t_\psi/2} +\int_0^{t_\psi} e^{-(t_\psi-t')/2} x_f^2 dt'=y_f +(y_p-y_f)e^{-t_\psi/2}$, where we have used $y_f=2x_f^2$.  Now $t_\psi=\Delta(y_\psi,x_p,x_f)<\Delta(y_f,x_p,x_f)$.  Therefore, $y_\psi>\psi_2(t_\psi)>y_b$, where
\begin{equation}
y_b=y_f+(y_p-y_f)\exp[-\Delta(y_f,x_p,x_f)/2],
\label{yb}
\end{equation}
and $y_b$ is an affine function of $y_p$.  Note that the right hand side of  (\ref{delcond2}) is a monotonic decreasing function of $y_\psi$.
Hence, (\ref{delcond2}) is guaranteed to hold if
\begin{equation}
\frac{x_r(1-x_f+y_f)(x_M-1-y_r)}{(1-x_M+y_f)(x_r-1-y_r)}>\frac{(1+y_b-x_f)x_p}{1+y_b-x_p},
\label{delcond3}
\end{equation}
which is a condition on tolerance for the initial value $(x_p,y_p)$ of $\psi(t)$ in terms of the initial value and extremal points of the reference trajectory $\phi(t)$.
Finally, we note that condition (\ref{delcond3}) is also applicable  for initial condition (b) or (c).  In those cases, set $x_M=x_r$.\\

\begin{remark}
If $y_p$ is increased for fixed $x_p$, then $y_\psi$ increases, such that the right hand side of (\ref{delcond2}) decreases.  Thus, the larger $y_p$ is, the more likely it is that (\ref{delcond2}) is satisfied.
\end{remark}\\

\begin{example}
Let $\dot y=rx-y$, $r>0$ and consider the following general equations as possibilities for $\dot x = f(x,y)$:
\begin{eqnarray}
\dot x = f(x,y) &=& \frac{ax^n}{1+by} - cx\\
\dot x = f(x,y) &=& ax-by^n\\
\dot x = f(x,y) &=& \frac{ax^n}{1+by^m} - cx,
\end{eqnarray}
where $a,b,c>0$, $n,m\in \mathbb{Z}^+$ and $x,y \geq0$.
\end{example}\\

Each of the above equations models inhibition of $x$ by $y$, with $f_y<0$ in the first quadrant, implying that the entire first quadrant is a region of inhibition.  Assuming parameters are chosen so that $(0,0)$ is a stable fixed point, results will be completely analogous to those in Example \ref{EX: Isoclines_NL_Example2}.\\

More diverse possibilities arise when $f_y \ge 0$ on at least a subset
of the first quadrant.
For example, suppose that $f(x,y)$ is the product of two inhibitory terms, such as
\[
f(x,y) = (ax+by)( \frac{cx}{1+dy}+h),
\]
with $b<0$ and $a,c,d>0$.
Indeed,
\[
sgn(f_y) = sgn(cx(b-adx)+bh(1+dy)^2).
\]
If $h>0$, then $f_y<0$ for all $(x,y) \in \mathbb{R}^{2+}$, as in the previous example.
If, however, $h<0$, then $f_y$ changes signs in $\mathbb{R}^{2+}$.\\

\begin{example}
\label{EX: Isoclines_NL_Example1}Consider the nonlinear system \begin{equation} \left. \begin{array}{ccccc}
\dot{x} & = & f(x,y) & = & (.5x-y)\left( \frac{0.1x}{1+y}-1\right) \\
\dot{y} & = & g(x,y) & = & 0.4x-y\text{,}
\end{array}
\right\}.
\label{System: NLExample1}
\end{equation}
\end{example} \\
with $(0,0)$ as a stable node.
The isoclines for this system are the family of curves given by the equations:
\begin{eqnarray}
y^{(1)} &=&\frac{3}{10}x-\frac{1}{2}+\frac{1}{2}C+\frac{1}{10}\sqrt{
4x^{2}+20x+30xC+25+50C+25C^{2}}\text{,}
\label{EQN: Iso_NL_ex1_Top curves} \\
y^{(2)} &=&\frac{3}{10}x-\frac{1}{2}+\frac{1}{2}C-\frac{1}{10}\sqrt{
4x^{2}+20x+30xC+25+50C+25C^{2}}\text{,}
\label{EQN: Iso_NL_ex1_Bottom curves}
\end{eqnarray}
where $C\in R$. In Figure \ref{FIG: Isoclines_Ex1}, the isoclines are drawn for various values of $C\in \lbrack -2.5,5.0]$, in increments of $0.25$.
For each $C\in \mathbb{R}$, the two curves defined by equations (\ref{EQN: Iso_NL_ex1_Top curves}) and (\ref{EQN: Iso_NL_ex1_Bottom curves}) together form a continuous curve. A thick black curve in the figure emphasizes the two parts, with equation (\ref{EQN: Iso_NL_ex1_Top curves}) forming the curves above and equation (\ref{EQN: Iso_NL_ex1_Bottom curves}) forming those beneath.  The equation of this curve, which looks linear in the first quadrant, is given by $y= -1+.2236 \sqrt{x(x+2)}$.

The portion of the first quadrant containing the top portions of the isoclines is not a region of inhibition, since for fixed $x$, $f$ is an increasing function of $y$ there.
However, the portion of the first quadrant containing the bottom portions of the isoclines is a region of inhibition, since $f$ is a decreasing function of $y$ there.
The curves given by the portion of the $x$-nullcline $(C=0)$ in the first quadrant are marked (red) to help delineate where the speed of the isoclines (i.e. $\dot x$) is positive or negative.

\begin{figure}[h]
\begin{center}
\includegraphics[width=4in]{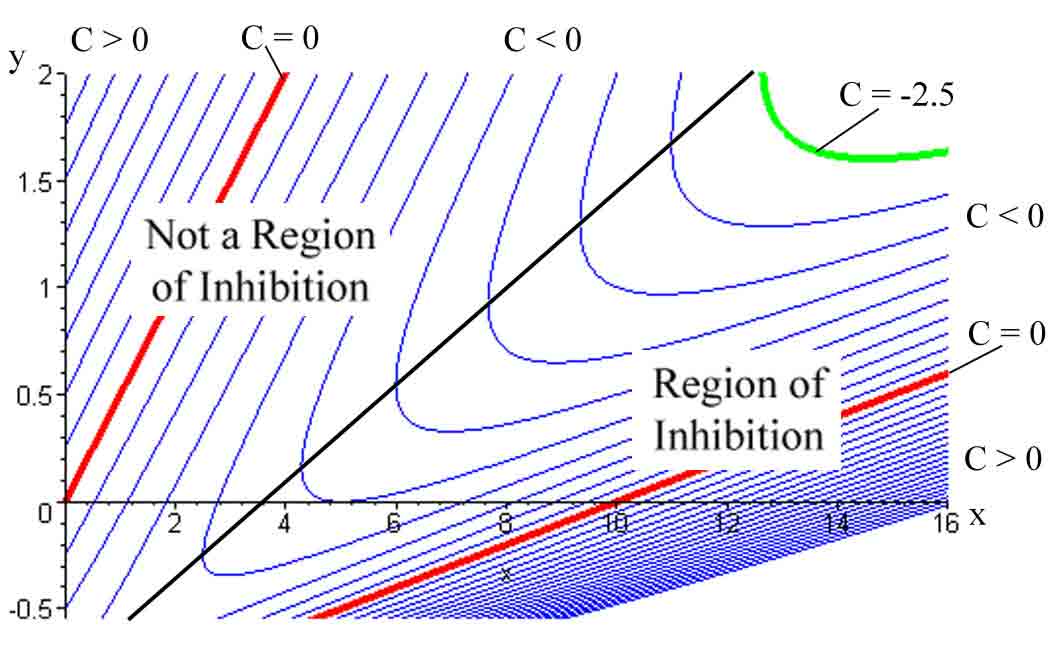}
\caption[Isoclines for Example \ref{EX: Isoclines_NL_Example1}.]{Isoclines for Example \ref{EX: Isoclines_NL_Example1}, drawn for various values of $C_1\in (-2.5,5)$, in increments of $0.25$. The thick black line marks the boundary of the region of inhibition.}
\label{FIG: Isoclines_Ex1}
\end{center}
\end{figure}

\begin{figure}[h]
\begin{center}
\includegraphics[width=4in]{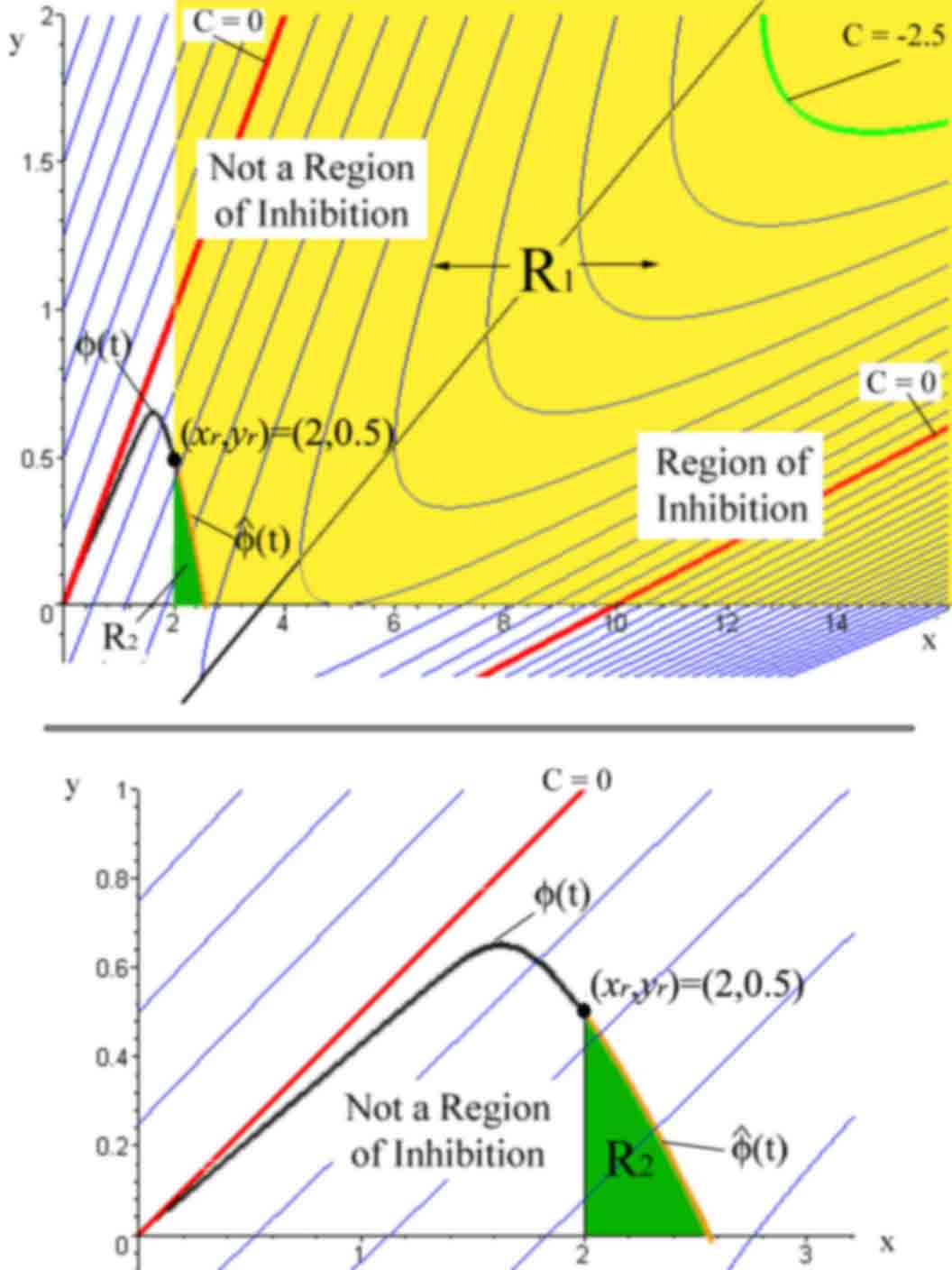}
\caption[Isoclines and important features in Example \ref{EX: Isoclines_NL_Example1}.]{Top: For Example \ref{EX: Isoclines_NL_Example1}, possible $(x_p,y_p)$ points fall in one of two regions: $R_2$, the green area plus its boundaries and $R_1$, the complement of $R_2$ with respect to $\Gamma^{x_r}_{(0,0)}$. Bottom: A close up of region $R_2$.}
  \label{FIG: Isoclines_Regions_Ex1}
\end{center}
\end{figure}

Figure \ref{FIG: Isoclines_Regions_Ex1} shows a specific solution, $\phi (t)$, that will be considered for this example. The following features appear in Figure \ref{FIG: Isoclines_Regions_Ex1}:

\begin{itemize}
\item $\phi (t)$ is the curve shown in black for initial condition $\phi(0)=(x_r,y_r)=(2,0.5)$.

\item The orange curve, denoted as $\hat{\phi}$, is the curve of points obtained by integrating $\phi (t)$\ backwards in time from $t=0$ to $t\approx -0.85$, at which time it intersects the $x$-axis at $\hat{x}\approx 2.5$.

\item Let $R_2$ be the region shown in green together with the boundaries made by (1) the line segment $\{(x,y)|x=2,0\leq y<0.5\}$, (2) the orange curve, $\hat{\phi}$, and (3) the $x$-axis.

\item Define the region $R_1$ to be the complement of $R_2$ in $\Gamma^{x_r}_{(0,0)}$, namely
$R_1 \defeq \Gamma^{x_r}_{(0,0)}\backslash R_2.$\\

\end{itemize}

Recall that every point $(x_p,y_p)$ will lie on or to the right of the line $x=x_r$, by $(A3)$. The regions $R_1$ and $R_2$ are formed so that for $(x_p,y_p) \in R_1$, $\psi$ will be bounded below by $\phi$ and for $(x_p,y_p)\in R_2\setminus \hat\phi$, $\psi$ will be bounded above by $\phi$. The graph of $\hat{\phi}$, in orange, creates a natural boundary (by uniqueness of solutions) between different classes of solutions $\psi(t)$.\\

Case 1: Let $(x_p,y_p)\in R_1$. Then, $\psi$ will be bounded below by $\phi$. Note that the graph of $\phi$ never enters the region of inhibition. Thus, any $(x_p,y_p) \in R_1$ does not produce tolerance with respect to $(x_r,y_r)$ by Proposition \ref{THM: Inhibition_1}.\\

Case 2: Let $(x_p,y_p)\in R_2\setminus \hat \phi$. The resulting $\psi$
will be bounded above by $\phi$. Thus, from Proposition \ref{THM: Inhibition_2}, since there are no regions of inhibition that contain both $\psi(t)$ and $\phi(t)$ for all $t\geq0$, tolerance may occur for $(x_p,y_p)$.
However, if $(x_p,y_p)$ lies \textit{on} the orange curve $\hat{\phi}$, then $\psi (t)$\ and $\phi (t)$\ are subsets of the same larger solution curve of the vector field (\ref{System: NLExample1}) and both $\phi_1(t)$ and $\psi_1(t)\rightarrow 0$ monotonically as $t\rightarrow \infty $. By Proposition \ref{Prop: Psi on Neg phi}, therefore, $(x_p,y_p)$ will not produce tolerance. In addition, by continuity, there exists an open ball, $B$, around each $(x_p,y_p)\in \hat \phi$, such that $(\tilde x_b,\tilde y_b)$ will not produce tolerance for all $(\tilde x_b,\tilde y_b)\in B$. Thus, the set of points which might produce tolerance is  a strict subset of region $R_2$. This set can be characterized more
extensively by two different arguments.

\begin{figure}
  \includegraphics[width=5in]{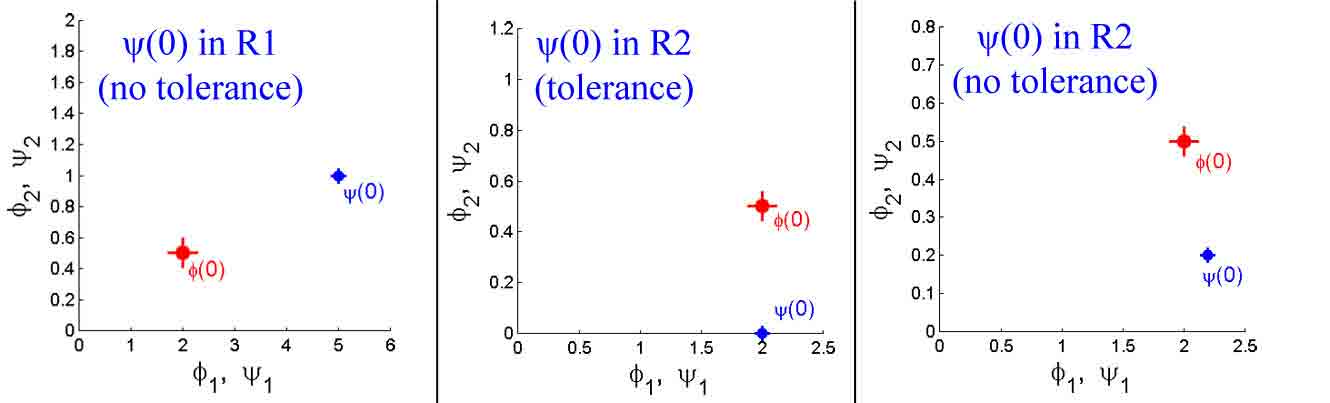}\\
  \caption{Animations for Example \ref{EX: Isoclines_NL_Example1}, showing the presence or absence of tolerance with respect to $\phi(0)=(2,0.5)$ for differing choices of $\psi(0)$. $\phi(0)$ is denoted by the large red dot and $\psi(0)$ is denoted by the smaller blue dot.  Given $\phi(0)=(2,0.5)$, $\psi(0)=(5,1)\in R_1$ does not produce tolerance $($Left$)$, $\psi(0)=(2,0) \in R_2$ produces tolerance $($Middle$)$, and $\psi(0)=(2.2,0.2)\in R_2$ does not produce tolerance $($Right$)$.}\label{FIG: ScreenShotEX319}
\end{figure}

First, it is clear that tolerance occurs if $(x_p,y_p)=(2.0,0)$, since
$f(2,0)<f(2,0.5)<0$ (see the animation associated with the middle panel of Figure \ref{FIG: ScreenShotEX319}).
Thus, tolerance occurs for all $(x_p,y_p)$ in a ball around $(2,0)$, intersected with $\Gamma^{x_r=2}_{(0,0)}$. 
The speed $f(x_p,0)$ becomes monotonically less negative as $x_p$ increases toward 2.5, and tolerance does not occur for $(x_p,y_p) = (2.5,0)$ by Proposition \ref{Prop: Psi on Neg phi}.
Thus, tolerance occurs for $(x_p,0)$ for all $x_p \in [2,\bar{x}_p)$ for some
$\bar{x}_p \in (2,2.5)$.
Similarly, $f(2,y_p)$ becomes monotonically less negative as $y_p$ increases
from 0, where tolerance occurs, to 0.5, where it does not.
Hence, tolerance occurs for $(2,y_p)$ for all $y_p \in [0,\bar{y}_p)$ for some
$\bar{y}_p \in (0,0.5)$.
Therefore, there is a continuous curve connecting $(\bar{x}_p,0)$ to
$(2,\bar{y}_p)$, call it $C_T$, such that tolerance occurs exactly when
$(x_p,y_p)$ is in the interior of the region bounded by $\{ x=2 \}$,
$\{ y=0 \}$, and $C_T$.

Second, to definitively establish that tolerance occurs for some specific $(x_p,y_p) \in R_2$, time interval estimates for specific trajectories must be made, as done in Example \ref{EX: Isoclines_NL_Example2}.  Figure \ref{FIG: ScreenShotEX319} provides links to three animations for Example \ref{EX: Isoclines_NL_Example1} using various choices of $\psi(0)$ from the different regions shown in Figure \ref{FIG: Isoclines_Regions_Ex1}.\\

\begin{example}
\label{EX: Isoclines_NL_Example3}Consider the nonlinear system:
\begin{equation}
\left.
\begin{array}{ccccl}
\dot{x} & = & f(x,y) & = & x\left(\frac{1+y^2}{1-y+y^2}-1.9 \right)\\
\dot{y} & = & g(x,y) & = & x-y
\end{array}
\right\}
\label{System: NLExample3}
\end{equation}
\end{example}\\
The isoclines for this system are the family of curves given by the
equations
\begin{eqnarray}
y^{(1)} &=&\frac{.5\left(19x+10C+\sqrt{37x^2-340xC-300C^2}\right)}{9x+10C}\text{,}
\label{EQN: Iso_NL_ex3_Top curves} \\
y^{(2)} &=&\frac{.5\left(19x+10C-\sqrt{37x^2-340xC-300C^2}\right)}{9x+10C}\text{,}
\label{EQN: Iso_NL_ex3_Bottom curves}
\end{eqnarray}
where $C\in R$. In Figure \ref{FIG: Isoclines_Ex3}, the isoclines are drawn in increments of $0.1$ for values of $C\in \lbrack -1.2,0]$ and in increments of $0.01$ for $C\in \lbrack 0,1]$. For $C\in \lbrack 0,1]$, the two curves defined by equations (\ref{EQN: Iso_NL_ex3_Top curves}) and (\ref{EQN: Iso_NL_ex3_Bottom curves}) together form a continuous curve. The black line, $y=1$, in the figure emphasizes the two parts, with equation (\ref{EQN: Iso_NL_ex3_Top curves}) forming the curves above and equation (\ref{EQN: Iso_NL_ex3_Bottom curves}) forming those beneath.

\begin{figure}[h]
\begin{center}
\includegraphics[width=4in]{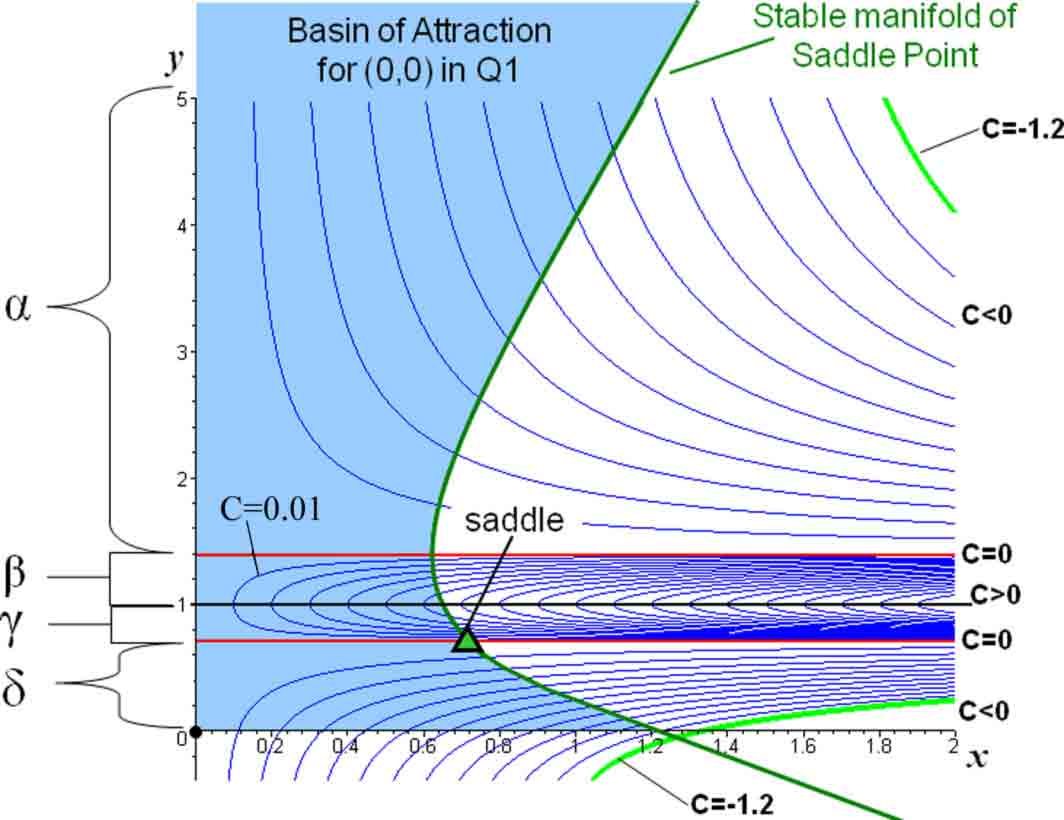}
\caption[Isoclines for Example \ref{EX: Isoclines_NL_Example3}.]{Isoclines for Example \ref{EX: Isoclines_NL_Example3}, drawn for various values of $C \in (-1.2,1)$.}
\label{FIG: Isoclines_Ex3}
\end{center}
\end{figure}

A saddle exists at $(0.72,0.72)$. The stable manifold of this saddle point forms a boundary for the basin of attraction of $(0,0)$, $\Gamma_{(0,0)}$.  The blue shaded region in Figure \ref{FIG: Isoclines_Ex3} shows the subset of $\Gamma_{(0,0)}$ in the first quadrant. A third fixed point (stable spiral, not labeled) in the first quadrant is located at $(1.4, 1.4)$, outside of $\Gamma_{(0,0)}$.  The $x$-nullclines $(C=0)$ are marked (red) to help delineate where the speeds associated with the isoclines (i.e. $\dot x$) are positive or negative.

We define several disjoint subregions (see Figure \ref{FIG: Isoclines_Ex3}) of the basin of attraction of $(0,0)$ in the first quadrant, as follows:
\begin{itemize}
\item $\alpha$ - above (and including) the top component of the $C=0$ isocline,
\item $\beta$ - below the top component of the $C=0$ isocline and above (and including) the line $y=1$,
\item $\gamma$ - below the line $y=1$ and above (and including) the bottom component of the $C=0$ isocline, and
\item $\delta$ - below the bottom component of the $C=0$ isocline.\\
\end{itemize}

These subregions are relevant because $C$ varies nonmonotonically in $y$ for this example and are defined to assist with identifying regions of inhibition. If looked at separately, subregions $\alpha$ and $\beta$ are both regions of inhibition and subregions $\gamma$ and $\delta$ are not regions of inhibition. However, additional complications may arise if $\phi$ and $\psi$ are not in the same subregion on some time interval.

\begin{figure}[h]
\begin{center}
\includegraphics[width=4in]{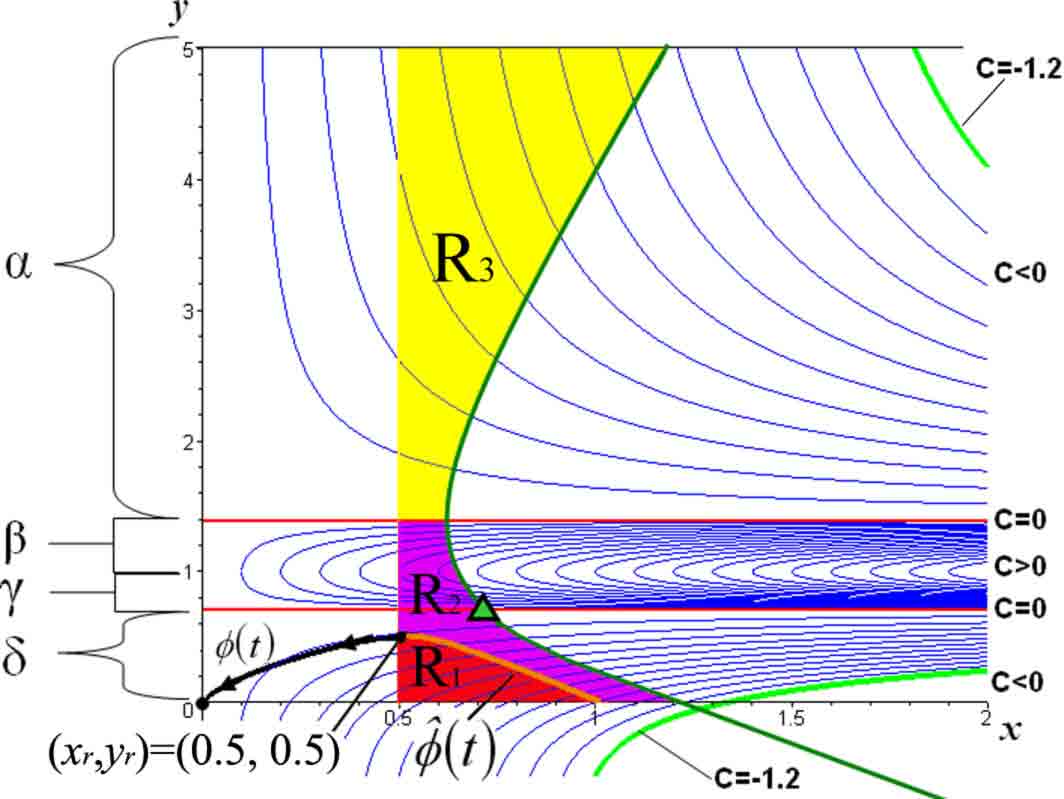}
\caption[Isoclines and important features in Example \ref{EX: Isoclines_NL_Example3}.]{Top: For Example \ref{EX: Isoclines_NL_Example3}, possible $(x_p,y_p)$ points fall in one of three regions: $R_1$, the red area plus its boundaries, $R_2$, the magenta region, and $R_3$, the yellow region.}
  \label{FIG: Isoclines_Regions_Ex3}
\end{center}
\end{figure}

Figure \ref{FIG: Isoclines_Regions_Ex3} shows one specific solution, $\phi (t)$ with $\phi(0)=(0.5,0.5)$, that will be considered for this example. The following features are also a part of Figure \ref{FIG: Isoclines_Regions_Ex3}:

\begin{itemize}
\item The orange curve, denoted as $\hat{\phi}$, is the curve of points obtained by integrating $\phi (t)$\  backwards in time from $t=0$ to $t\approx -1.75$, at which time it intersects the $x$-axis at $\hat{x}\approx 1.0$.

\item $R_1$ is the region shown in red together with the boundaries made by (1) the line segment $\{(x,y)|x=0.5,0\leq y\leq0.5\}$, (2) the orange curve $\hat{\phi}$, and (3) the $x$-axis.

\item $R_2$ is the region shown in magenta, defined as $R_2 \defeq \Gamma^{x_r}_{(0,0)}\setminus (\alpha\cup R_1)$.

\item $R_3$ is the region shown in yellow to be $R_3 \defeq \Gamma^{x_r}_{(0,0)} \cap \alpha$.\\
\end{itemize}

As usual, we consider points $(x_p,y_p)$ that lie on or to the right of the line $\{ x=x_r \}$.
 The region $R_1$ is formed so that for $(x_p,y_p) \in R_1\setminus \hat\phi$, $\psi$ will be bounded above by $\phi$. For $(x_p,y_p)\in R_2 \cup R_3$, $\psi$ will be bounded below by $\phi$. The graph of $\hat{\phi}$, in orange, creates a natural boundary (by uniqueness of solutions) for $\psi(t)$, as in the previous example.\\

Case 1: Let $(x_p,y_p)\in R_1\setminus \hat \phi$. Then, $\psi$ will be bounded above by $\phi$. Thus, from Proposition \ref{THM: Inhibition_2}, since there are no regions of inhibition that contain both $\psi(t)$ and $\phi(t)$ for all $t\geq0$, $(x_p,y_p)$ might produce tolerance. This case is very similar to that
considered in the previous example.
Indeed, it is clear that tolerance occurs if $(x_p,y_p)=(0.5,0)$, while
tolerance does not occur if $(x_p,y_p)$ lies on $\hat{\phi}$, by Proposition \ref{Prop: Psi on Neg phi}. Again, there will be a continuous curve connecting $\{ (x,y): x=0.5, 0 \leq y \leq 0.5 \}$ to $\{ (x,y) : 0.5 \leq x \leq 1 , y=0 \}$
such that tolerance occurs for all $(x,p)$ in $R_1$ below this curve and does not occur in $R_1$ above this curve.
Time interval estimates are necessary to prove that tolerance occurs or does not occur for specific choices of $(x_p,y_p)$ in $R_1$.\\

Case 2: Let $(x_p,y_p)\in R_2$. Then, $\phi \subset \delta$ and $\psi$ will be bounded below by $\phi$.
Note that $\gamma$ is not a region of inhibition and that $\beta$, although a region of inhibition by itself, has $f>0$, such that no tolerance can occur before $\psi$ enters $\delta$.
 But $\delta$ is not a region of inhibition, and hence from Proposition \ref{THM: Inhibition_1}, any $(x_p,y_p) \in R_2$ does not produce tolerance with respect to $(x_r,y_r)$.\\

Case 3:  Let $(x_p,y_p) \in R_3$.  Since $C<0$ in $R_3$, it is possible in this case that tolerance will occur before $\psi$ leaves $\alpha$.
Alternatively, suppose that this does not happen.
After $\psi$ leaves $\alpha$, it enters $\beta, \gamma$, and finally $\delta$ as it converges toward $(0,0)$.
In theory, tolerance could occur after $\psi$ enters $\delta$.
However, $\psi$ is bounded below by $\phi$ and $\delta$ is not a region of inhibition.
Hence, as in Case 2, Proposition \ref{THM: Inhibition_1} implies that tolerance will not occur.
In summary, if $\phi(0) \in \delta$ and $(x_p,y_p) \in R_3$, then
either tolerance occurs before $\psi$ leaves $\alpha$ or it does not occur at all.

\begin{figure}[h]
\begin{center}
\includegraphics[width=4in]{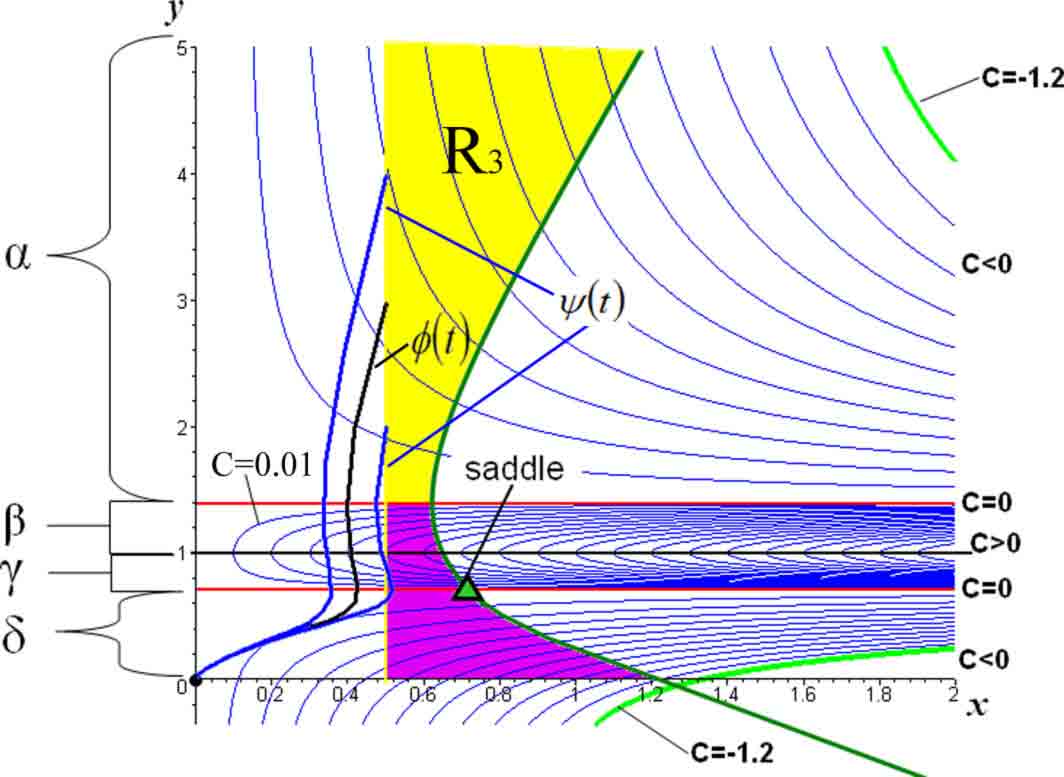}
\caption[]{Nonmonotonic convergence to $(0,0)$ in Example \ref{EX: Isoclines_NL_Example3}.}
  \label{FIG: ISO_NewICs}
\end{center}
\end{figure}

Using the same nonlinear system given by (\ref{System: NLExample3}), consider an alternative choice for $(x_r, y_r)$, namely one in $\alpha$. Such a choice demonstrates some additional complexities that can arise in this type of example.  Now, $\phi$ passes through regions where $f<0$, then $f>0$, and finally $f<0$ again as it converges to $(0,0)$.  For different $\psi$ trajectories, either bounded above or below by $\phi$ (see Figure \ref{FIG: ISO_NewICs}), there are different time intervals when tolerance cannot occur or might possibly occur, which can be inferred from the isoclines.

In the particular example shown, for the $\psi$ that is bounded below by $\phi$, tolerance cannot be ruled out in any region. In particular, let $x_M$ denote the $x$-value where $\phi$ intersects the
$x$-nullcline branch that forms the boundary between $\gamma$ and $\delta$.
If $\psi_1(t) < x_M$ when $\phi$ passes from $\alpha$ to $\beta$, then tolerance is guaranteed to occur.
On the other hand, for the $\psi$ that is bounded above by $\phi$, tolerance is only possible after $\psi$ enters $\delta$.  Figure \ref{FIG: ScreenShotEX320} provides links for two animations for Example \ref{EX: Isoclines_NL_Example3} using $\phi(0)=(0.5, 2)$ in Region $R_3$ and two choices of $\psi(0)$ also in Region $R_3$, similar to those shown in Figure \ref{FIG: ISO_NewICs}.\\

\begin{figure}
  \includegraphics[width=5in]{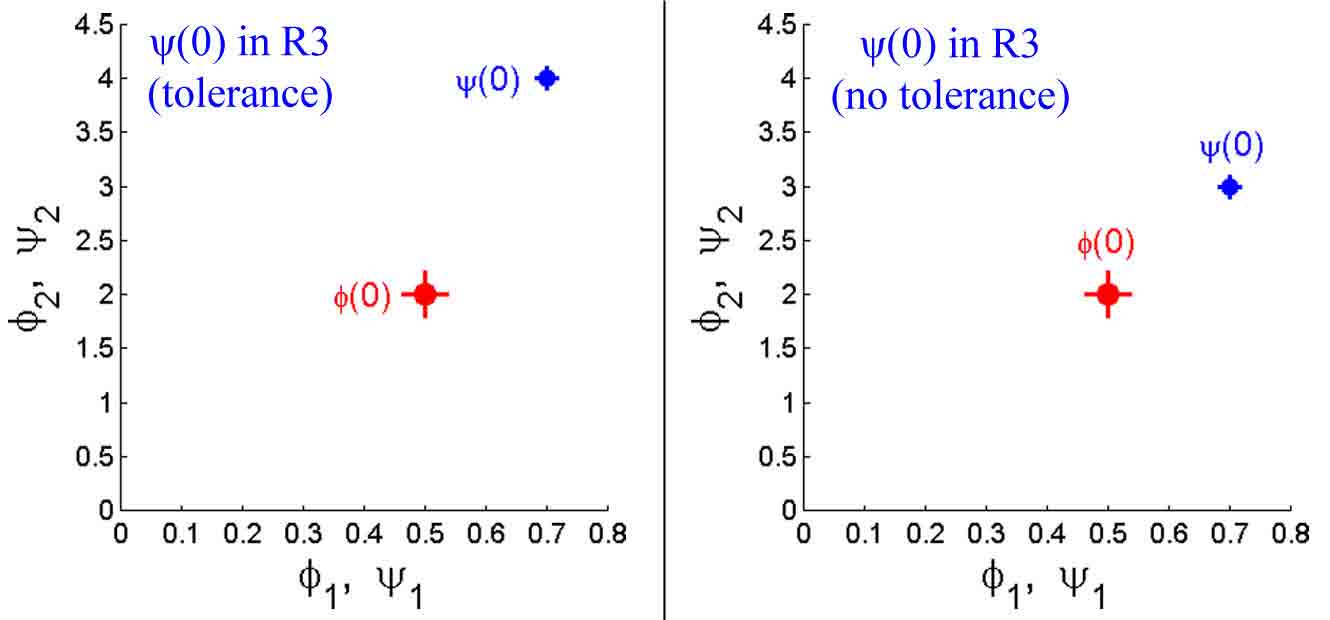}\\
  \caption{Animations for Example \ref{EX: Isoclines_NL_Example3}, using an alternative choice for $\phi(0)$, namely in region $\alpha$, as shown in Figure \ref{FIG: ISO_NewICs}. The presence or absence of tolerance with respect to $\phi(0)=(0.5, 2)$ is shown for differing choices of $\psi(0)$ in Region $R_3$. $\phi(0)$ is denoted by the large red dot and $\psi(0)$ is denoted by the smaller blue dot.  Given $\phi(0)=(0.5, 2)$, $\psi(0)=(0.7,4)\in R_3$ produces tolerance $($Left$)$ and $\psi(0)=(0.7,3) \in R_3$ does not produces tolerance $($Right$)$.}\label{FIG: ScreenShotEX320}
\end{figure}

\section{Tolerance in Linear ODE systems}

\label{Section: 2DLinear}
The previous sections have established that it is sometimes  difficult to make precise general statements about tolerance.  However, in the case of linear systems, we can fully characterize the occurrence of tolerance for equation (\ref{system}). In this section, we derive a complete set of necessary and sufficient conditions for the existence of tolerance in 2D linear ODE systems.

Consider the linear system
\begin{equation}
\dot{x}=Ax\text{,}  \label{system_linear}
\end{equation}
where $A\in M^{2x2}$, $x\in \mathbb{R}^{2+}=[0,\infty )\times \lbrack
0,\infty )$. \ Throughout this section, we will assume as before that:

\begin{description}
\item [(A1)] $(0,0)$ is a stable fixed point of $($\ref{system_linear}$)$,
the eigenvalues of which are real and negative.

\item [(A2)] $\phi (t)$ and $\psi (t)$ are nonnegative for all $t\geq 0$ and both $(x_r,y_r)$ and $(x_p,y_p)$ lie in the basin of attraction for $(0,0)$ in the first quadrant, $\Gamma^+_{(0,0)}$.

\item[(A3)] $x_p \geq x_r$.
\end{description}

Let $\lambda_1$ and $\lambda_2$ be the real, negative eigenvalues of $A$. To arrive at necessary and sufficient conditions for the existence of tolerance, there are two cases that must be considered.
The first case is that $A$ has distinct eigenvalues, $\lambda_1\neq \lambda_2$.  The other case is that $A$ has identical eigenvalues, $\lambda _1=\lambda _2 \defeq \lambda <0$. For each of these cases, there are subcases to consider as well.

\subsection{Case 1: $\lambda _{1}\neq \lambda _{2}$}

For this case, where $\lambda _{1}$ and $\lambda _{2}$ are distinct, negative eigenvalues of $A$, assume without loss of generality that $\lambda_{2}<\lambda _{1}<0$. Let $v$ be an eigenvector corresponding to $\lambda _{1}$, and let
$w$ be an eigenvector corresponding to $\lambda _{2}$. Since $\lambda _{1}$ and $\lambda_{2}$ are distinct, $v$ and $w$ are linearly independent.
Thus, any initial condition can be uniquely written as a linear combination of $v$ and $w$. In particular, $(x_r,y_r)=c_{1}v+c_{2}w=(c_{1}v_{1}+c_{2}w_{1},c_{1}v_{2}+c_{2}w_{2})$, with $c_{1},c_{2}\in \mathbb{R}$. Then, the solution $\phi (t)$ to the initial value problem (IVP) $\dot{x}=Ax$, $\phi (0)=(x_r,y_r)$ is
\begin{equation}
\phi (t)=c_{1}ve^{\lambda _{1}t}+c_{2}we^{\lambda
_{2}t}=(c_{1}v_{1}e^{\lambda _{1}t}+c_{2}w_{1}e^{\lambda
_{2}t},c_{1}v_{2}e^{\lambda _{1}t}+c_{2}w_{2}e^{\lambda _{2}t}).
\label{Case1: Phi(t)}
\end{equation}
Similarly, consider the initial condition $(x_p,y_p)$,
 which can be uniquely written as $(x_p,y_p)=d_1v+d_2w=(d_1v_1+d_2w_1,d_1v_2+d_2w_2)$, with $d_1,d_2\in \mathbb{R}$. The solution $\psi (t)$ to the IVP $\dot{x}=Ax$, $\psi(0)=(x_p,y_p)$ is
\begin{equation}
\psi (t)=d_{1}ve^{\lambda _{1}t}+d_{2}we^{\lambda
_{2}t}=(d_{1}v_{1}e^{\lambda _{1}t}+d_{2}w_{1}e^{\lambda
_{2}t},d_{1}v_{2}e^{\lambda _{1}t}+d_{2}w_{2}e^{\lambda _{2}t}).
\label{Case1: Psi(t)}
\end{equation}
Since we know $x_p\geq x_r$ by (A3), we have that
\begin{equation}
d_{1}v_{1}+d_{2}w_{1}\geq c_{1}v_{1}+c_{2}w_{1}\text{.}\
\label{Case1: ICCondition}
\end{equation}
We will consider three subcases for Case 1: (a) $v_{1}=0$ and $w_{1}=1$ (b) $v_{1}=1$ and $w_{1}=0$ and (c) $v_{1}=$ $w_{1}=1$.

\subsubsection{Case 1a: $v_{1}=0$ and $w_{1}=1$}

For this case, (\ref{Case1: ICCondition}) becomes
\begin{equation}
d_{2}\geq c_{2}.  \label{Case1a: ICCondition}
\end{equation}%
Consider the difference between $\phi _{1}(t)$ and $\psi _{1}(t)$. \ Using
equations (\ref{Case1: Phi(t)})\ and (\ref{Case1: Psi(t)}) as well as
 $v_{1}=0$ and $w_{1}=1$, we have%
\begin{equation*}
\phi _{1}(t)-\psi _{1}(t) = (c_{2}-d_{2})e^{\lambda _{2}t}\text{.}
\end{equation*}%
By (\ref{Case1a: ICCondition}), we have that $(c_{2}-d_{2})\leq 0$. \ Thus,
because $e^{\lambda _{2}t}>0$ for all $t\geq 0$, we have that $\phi
_{1}(t)-\psi _{1}(t)\leq 0$ for all $t\geq 0$. \ Therefore, the following
result has been shown.\\

\begin{proposition}
Assume $($A1$)$, $($A2$)$, $($A3$)$ and that $\lambda _2<\lambda _1<0$. Given $\langle(x_r,y_r),(x_p,y_p)\rangle$, if $v_1=0$ and $w_1=1$ for eigenvectors $v$ and $w$ of $\lambda_1$ and $\lambda_2$, respectively, then $\dot{x}=Ax$ does not exhibit tolerance for $\langle(x_r,y_r),(x_p,y_p)\rangle$.
\end{proposition}

\subsubsection{Case 1b: $v_{1}=1$ and $w_{1}=0$}

For this second subcase of Case 1, (\ref{Case1: ICCondition}) becomes
\begin{equation}
d_1\geq c_1.  \label{Case1b: ICCondition}
\end{equation}%
Using equations (\ref{Case1: Phi(t)})\ and (\ref{Case1: Psi(t)}) as well as
$v_{1}=1$ and $w_{1}=0$, we have
\begin{equation*}
\phi _{1}(t)-\psi _{1}(t) =  (c_{1}-d_{1})e^{\lambda _{1}t}\text{.}
\end{equation*}%
By (\ref{Case1b: ICCondition}), we have that $(c_{1}-d_{1})\leq 0$.
Thus, we conclude
$\phi_{1}(t)-\psi _{1}(t)\leq 0$ for all $t\geq 0$, and the following result has been
shown.\\

\begin{proposition}
Assume $($A1$)$, $($A2$)$, $($A3$)$ and that $\lambda_2<\lambda_1<0$. \ Given $\langle(x_r,y_r),(x_p,y_p)\rangle$, if $v_1=1$ and $w_1=0$ for eigenvectors $v$ and $w$ of $\lambda_1$ and $\lambda_2$, respectively, then $\dot{x}=Ax$ does not exhibit tolerance for $\langle(x_r,y_r),(x_p,y_p)\rangle$.
\end{proposition}

\subsubsection{Case 1c: $v_{1}=$ $w_{1}=1$}

Unlike Cases 1a and 1b, tolerance is a possibility in case 1c.  Proposition \ref{THM: Linear Tol 3} below states necessary and sufficient conditions on the coefficients of the solutions $\phi$ and $\psi$ in order for tolerance to be exhibited and also specifies the precise time value beyond which tolerance is exhibited, when it occurs.\\

\begin{proposition}
\label{THM: Linear Tol 3}Assume $($A1$)$, $($A2$)$, $($A3$)$, $\lambda_2<\lambda_1<0$, and $v_1=$ $w_1=1$ for eigenvectors $v$ and $w$ of $\lambda_1$ and $\lambda_2$, respectively. Given $\langle(x_r,y_r),(x_p,y_p)\rangle$, then there exists $T>0$ such that (\ref{system_linear}) will exhibit tolerance for all $t>T$ if and only if $c_1>d_1$ and $c_2<d_2$. Furthermore,
\begin{equation}
T=\frac{\ln [(d_{2}-c_{2})/(c_{1}-d_{1})]}{\lambda _{1}-\lambda _{2}}\text{.}
\label{EQN: Tcase1c}
\end{equation}
\end{proposition}

\begin{proof}
 Necessary Conditions. Assume that
$c_1 \leq d_1$. Since $v_{1}=w_{1}=1$, we may rewrite (\ref{Case1: ICCondition}) as
\begin{equation}
d_{1}+d_{2}\geq c_{1}+c_{2}.  \label{Condition}
\end{equation}%
Consider the difference between $\phi _{1}(t)$ and $\psi _{1}(t)$. Using (\ref{Case1: Phi(t)}), (\ref{Case1: Psi(t)}), $v_{1}=w_{1}=1$, and (\ref{Condition}), we have
\begin{eqnarray*}
\phi _{1}(t)-\psi _{1}(t)
&=&(c_{1}-d_{1})e^{\lambda _{1}t}+(c_{2}-d_{2})e^{\lambda _{2}t} \\
&\leq &(c_{1}-d_{1})e^{\lambda _{1}t}+(d_{1}-c_{1})e^{\lambda _{2}t} \\
&=&(c_{1}-d_{1})(e^{\lambda _{1}t}-e^{\lambda _{2}t}).
\end{eqnarray*}%
Since $\lambda _{2}<\lambda _{1}<0$ and
$c_1 \leq d_1$, it follows that $\phi _{1}(t)-\psi _{1}(t)\leq 0$, which means that $\psi _{1}(t)\geq \phi _{1}(t)$ for all $t\geq 0$. Hence, tolerance cannot
be exhibited for $c_1 \leq d_1$.
Similarly, it can be shown that (\ref{system_linear}) cannot exhibit tolerance for $c_2 \geq d_2$.
Thus, $c_1 > d_1$ and $c_2 < d_2$ are both necessary conditions for tolerance.

Sufficient Conditions. Assume that $c_1>d_{1}$ and $c_2<d_2$ both hold. Using (\ref{Case1: Phi(t)}), (\ref{Case1: Psi(t)}), and $v_1=w_1=1$, we have
\begin{eqnarray*}
\phi _{1}(t)-\psi _{1}(t) &=&(c_{1}-d_{1})e^{\lambda _{1}t}+(c_{2}-d_{2})e^{\lambda _{2}t}\text{.}\\
&=& \left( e^{(\lambda _{1}-\lambda _{2})t}+\frac{%
c_{2}-d_{2}}{c_{1}-d_{1}}\right) e^{\lambda _{2}t}(c_{1}-d_{1}).
\end{eqnarray*}%
By assumption, $(c_{1}-d_{1})>0$ and $(c_{2}-d_{2})<0$, and thus
\begin{equation*}
e^{\lambda _{2}t}(c_{1}-d_{1})>0 \; \mbox{ and } \; \frac{(c_{2}-d_{2})}{(c_{1}-d_{1})}<0.
\end{equation*}

\pagebreak

Therefore,

\begin{eqnarray*}
\phi _{1}(t)-\psi _{1}(t) &=&\left( e^{(\lambda _{1}-\lambda _{2})t}+\frac{%
(c_{2}-d_{2})}{(c_{1}-d_{1})}\right) e^{\lambda _{2}t}(c_{1}-d_{1}) >0 \\
&\Leftrightarrow &\left( e^{(\lambda _{1}-\lambda _{2})t}+\frac{(c_{2}-d_{2})%
}{(c_{1}-d_{1})}\right) >0 \\
&\Leftrightarrow &e^{(\lambda _{1}-\lambda _{2})t}>\frac{(d_{2}-c_{2})}{%
(c_{1}-d_{1})} \\
&\Leftrightarrow &t>\frac{\ln[(d_{2}-c_{2})/(c_{1}-d_{1})]}{
(\lambda _{1}-\lambda _{2})}
\end{eqnarray*}
\end{proof}

\subsection{Case 2: $\lambda_1=\lambda_2=\lambda <0$} In this case, $\lambda$ has either a one- or two-dimensional eigenspace. Thus, two subcases need to be considered.

\subsubsection{Case 2a: $\lambda_1=\lambda_2=\lambda <0$ and $\lambda $ has a two-dimensional eigenspace} For this case, $\lambda $ is an eigenvalue of $A$ with multiplicity two for which two linearly independent eigenvectors can be found. Let $v$ and $w$ be linear independent eigenvectors of $\lambda $. Then, any initial condition can be uniquely written as a linear combination of $v$ and $w$. For the initial condition $(x_r,y_r)$, we may write $(x_r,y_r)=c_1v+c_2w=(c_1v_1+c_2w_1,c_1v_2+c_2w_2)$, with $c_1,c_2\in \mathbb{R}$. Thus, the solution, $\phi (t)$, to the IVP $\dot{x}=Ax$, $\phi(0)=(x_r,y_r)$ is
\begin{equation}
\phi (t)=c_1ve^{\lambda t}+c_2we^{\lambda t}=(c_1v_1+c_2w_1,c_1v_2+c_2w_2)e^{\lambda t}=(x_re^{\lambda t},y_re^{\lambda t}).
\label{Case2a: phi}
\end{equation}
Similarly, consider the initial condition $(x_p,y_p)$, which may also uniquely be written as $(x_p,y_p)=d_1v+d_2w=(d_1v_1+d_2w_1,d_1v_2+d_2w_2)$, with $d_1,d_2\in \mathbb{R}$. The solution $\psi (t)$ to the IVP $\dot{x}=Ax$, $\psi(0)=(x_p,y_p)$ is
\begin{equation}
\psi (t)=d_1ve^{\lambda t}+d_2we^{\lambda t}=(d_1v_1+d_2w_1,d_1v_2+d_2w_2)e^{\lambda t}=(x_pe^{\lambda t},y_pe^{\lambda t}).
\label{Case2a: psi}
\end{equation}

Consider the difference between $\phi_1(t)$ and $\psi_1(t)$. Using  (\ref{Case2a: phi}),  (\ref{Case2a: psi}), $(A3)$ and the fact that $e^{\lambda t}>0$ for all $t\geq 0$ we have that :
\begin{equation*}
\phi_1(t)-\psi_1(t) = x_re^{\lambda t}-x_pe^{\lambda t}= (x_r-x_p)e^{\lambda t}\leq 0.
\end{equation*}
Thus, $\psi_1(t)\geq \phi_1(t)$ for all $t\geq 0$, and the following has been shown:

\begin{proposition}
Assume $($A1$)$, $($A2$)$, $($A3$)$ and that $\lambda_1=\lambda_2=\lambda <0$. Given $\langle(x_r,y_r),(x_p,y_p)\rangle$, if $\lambda $ has two linearly independent eigenvectors, then $\dot{x}=Ax$ cannot exhibit tolerance for $(x_r,y_r)$.
\end{proposition}

\subsubsection{Case 2b: $\lambda_1=\lambda_2=\lambda <0$ and $\lambda $ has a one-dimensional eigenspace} In this case, let $v$ be an eigenvector of $\lambda $. One solution to (\ref{system_linear}) is $x^{(1)}(t)=ve^{\lambda t}$. A second solution to (\ref{system_linear}) is $x^{(2)}(t)=vte^{\lambda t}+\bar{v}e^{\lambda t}$, where $\bar{v}$ is a generalized eigenvector satisfying $(A-\lambda I)\bar{v}=v$. The initial condition $(x_r,y_r)$ can be uniquely written as a linear combination of $v$ and $\bar{v}$,
\begin{equation*}
(x_r,y_r)=c_{1}v+c_{2}\bar{v}=(c_{1}v_{1}+c_{2}\bar{v}_{1},c_{1}v_{2}+c_{2}\bar{v}_{2})\text{,with }c_{1},c_{2}\in \mathbb{R}\text{.}
\end{equation*}

\noindent The solution $\phi (t)$ to the IVP $\dot{x}=Ax$, $\phi(0)=(x_r,y_r)$ is
\begin{eqnarray}
\phi(t) &=& c_{1}ve^{\lambda t}+c_{2}(vte^{\lambda t}+\bar{v}e^{\lambda t})
\notag \\&=&(c_{1}v_{1}e^{\lambda t}+c_{2}(v_{1}te^{\lambda t}+\bar{v}_{1}e^{\lambda t}),c_{1}v_{2}e^{\lambda t}+c_{2}(v_{2}te^{\lambda t}+\bar{v}_{2}e^{\lambda t}))\text{.}  \label{Case2: Phi(t)}
\end{eqnarray}
Similiary, the initial condition, $(x_p,y_p)$, can be uniquely written as a linear combination of $v$ and $\bar{v}$,
\begin{equation*}
(x_p,y_p)=d_{1}v+d_{2}\bar{v}=(d_{1}v_{1}+d_{2}\bar{v}_{1},d_{1}v_{2}+d_{2}\bar{v}_{2})\text{,with }d_{1},d_{2}\in \mathbb{R}\text{,}
\end{equation*}
and the solution $\psi(t)$ to the IVP $\dot{x}=Ax$, $\psi (0)=(x_p,y_p)$ is
\begin{eqnarray}
\psi (t) &=&d_{1}ve^{\lambda t}+d_{2}(vte^{\lambda t}+\bar{v}e^{\lambda t}) \notag \\
&=&(d_{1}v_{1}e^{\lambda t}+d_{2}(v_{1}te^{\lambda t}+\bar{v}_{1}e^{\lambda
t}),d_{1}v_{2}e^{\lambda t}+d_{2}(v_{2}te^{\lambda t}+\bar{v}_{2}e^{\lambda
t}))\text{.}  \label{Case2: Psi(t)}
\end{eqnarray}

The following proposition, given without the details of its proof, states the result for this case.\\

\begin{proposition}
\label{THM: Case2b}Let $\langle(x_r,y_r),(x_p,y_p)\rangle$. Assume $($A1$)$, $($A2$)$, $($A3$)$, and that $\lambda_1=\lambda_2=\lambda <0$. Suppose that $\lambda $ has a one-dimensional eigenspace. Let $v$ be an eigenvector of $\lambda $ and let $\bar{v}$ be a corresponding generalized eigenvector.
\begin{description}
\item{(i)} If $v_1=1$ and $\bar v_1=0$, then there exists $T>0$ such that (\ref{system_linear}) will exhibit tolerance for $\langle(x_r,y_r),(x_p,y_p)\rangle$ for all $t>T$ if and only if $c_1 \leq d_1$ and  $c_2>d_2$ both hold. Furthermore,
\begin{equation}
T=\frac{d_{1}-c_{1}}{c_{2}-d_{2}}\text{,}
\label{EQN: Tcase2}
\end{equation}
and the difference between $\phi _{1}(t)$ and $\psi _{1}(t)$ at $t>T$ will be less than or equal to $(c_{2}-d_{2})te^{\lambda t}$. Therefore, $\underset{t>T}{\max}\left\{ (c_{2}-d_{2})te^{\lambda t}\right\} =\frac{d_{2}-c_{2}}{\lambda e}$, which occurs at $t=\frac{-1}{\lambda }$, is the greatest degree of tolerance that is possible.\\
\item{(ii)} If $\bar v_1 \neq 0$, then (\ref{system_linear}) will not exhibit tolerance for $\langle(x_r,y_r),(x_p,y_p)\rangle$.
\end{description}
\end{proposition}

\subsection{Eigenvector Configurations and Regions of Tolerance}\label{section: EVCRegions}

Of the cases discussed above, only Cases 1c and 2b yield the possibility of tolerance. The results stated above give analytical conditions for the existence of
tolerance in terms of coefficients of general solutions to (\ref{system}).
We find that these results are more useful when they are recast geometrically.
To achieve this reformulation, we consider eigenvector configurations (EVC) that accommodate solutions that satisfy the nonnegativity requirement (A2).  Each such configuration is displayed in Figure \ref{FIG: EVC_Regions}. For each configuration, we subdivide the positive quadrant into regions and then, for $(x_r,y_r)$ in each region, determine precisely which locations for $(x_p,y_p)$ will lead to tolerance and which will not.  The results for all the eigenvector configurations shown in \ref{FIG: EVC_Regions} are summarized in Table \ref{Tab: EVCResults} and are illustrated in the figures referenced in the table.

\begin{figure}[h]
\begin{center}
\includegraphics[width=5in]{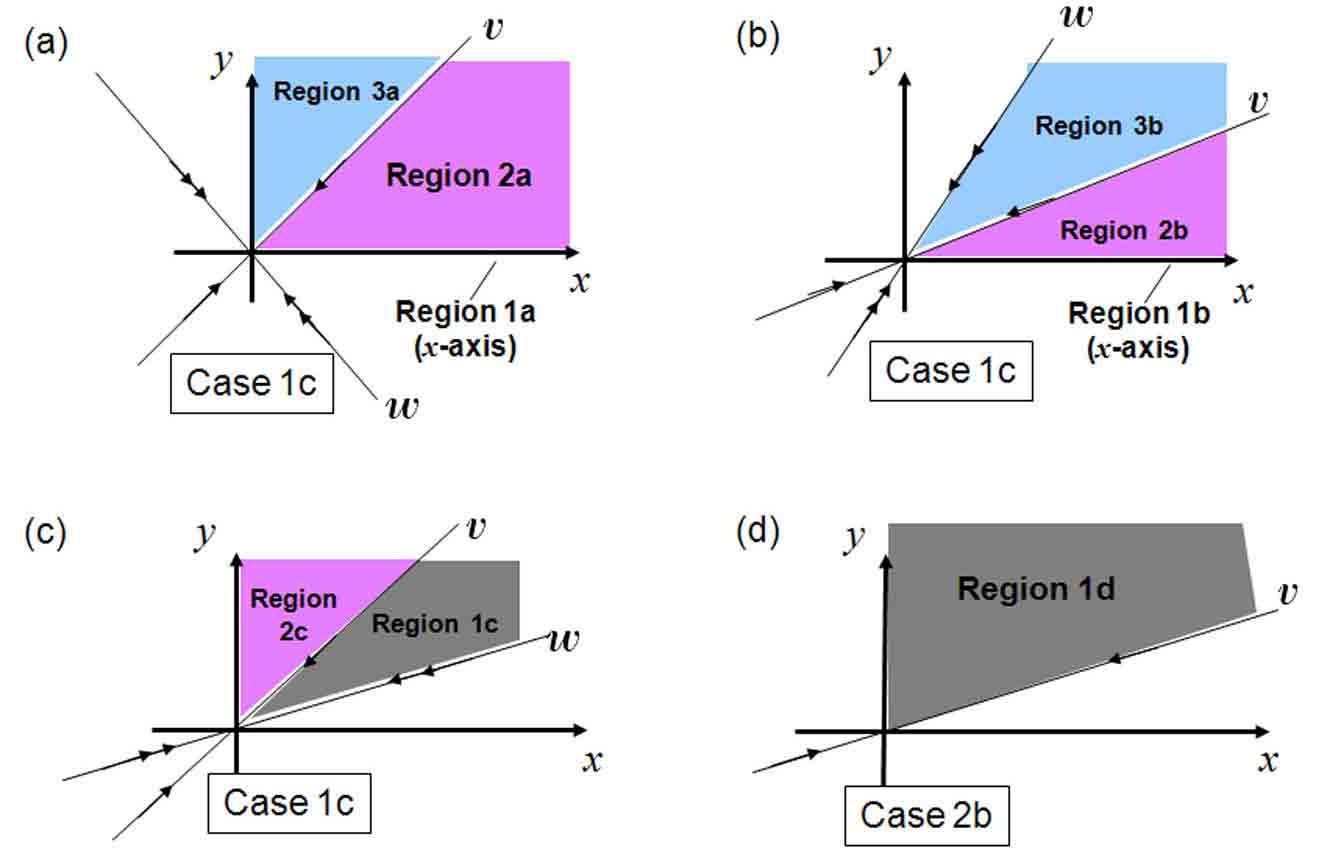}
\caption[Regions of interest in the first quadrant for four relevant eigenvector configurations]{Regions of interest in the first quadrant for four relevant eigenvector configurations.  Note that we label the weak eigenvector $v$ with one arrow and the strong eigenvector $w$ with two arrows. }\label{FIG: EVC_Regions}
\end{center}
\end{figure}

\begin{table}
\begin{center}
	\begin{tabular}{  p{2.5cm} | p{2cm} | p{2.8cm} | p{1.7cm} }
	\hline
	Eigenvector Configuration: & If $(x_r, y_r)$ is in Region: & Then, tolerance is produced by $(x_p, y_p)$ in: & Figure Reference: \\
\hline
\hline
\multirow{7}{*}{(a)  Figure \ref{FIG: EVC_Regions}a} & & & \\
 & 1a & None  & Figure \ref{FIG: AIO_EVC_a} (top) \\
 & & & \\
 & 2a  & Region $\mathbb{IV}_{2a}$  & Figure \ref{FIG: AIO_EVC_a} (middle) \\
 & & & \\
 & 3a  & Region $\mathbb{IV}_{3a}$  & Figure \ref{FIG: AIO_EVC_a} (bottom) \\
 & & & \\
 \hline
\multirow{7}{*}{(b)  Figure \ref{FIG: EVC_Regions}b} & & & \\
 & 1b  & Region $\mathbb{I}_{1b}$  & Figure \ref{FIG: AIO_EVC_b} (top) \\
 & & & \\
 & 2b  & Region $\mathbb{I}_{2b}$  & Figure \ref{FIG: AIO_EVC_b} (middle)\\
 & & & \\
 & 3b  & Region $\mathbb{II}_{3b}$  & Figure \ref{FIG: AIO_EVC_b} (bottom) \\
 & & & \\
\hline
\multirow{5}{*}{(c)  Figure \ref{FIG: EVC_Regions}c} & & & \\
 & 1c & Region $\mathbb{IV}_{1c}$  & Figure \ref{FIG: AIO_EVC_c} (top) \\
 & & & \\
 & 2c  & Region $\mathbb{IV}_{2c}$  & Figure \ref{FIG: AIO_EVC_c} (bottom) \\
 & & & \\
\hline
\multirow{3}{*}{(d) Figure \ref{FIG: EVC_Regions}d} & & & \\
 & 1d  & Region $\mathbb{IV}_{1d}$  & Figure \ref{FIG: EVC_a_2bii_2pics} \\
 & & & \\
\hline
\hline
	\end{tabular}
\end{center}
\caption[Table summarizing eigenvector configuration results]{Summary of tolerance results for eigenvector configurations shown in Figure \ref{FIG: EVC_Regions}}
\label{Tab: EVCResults}
\end{table}

\subsubsection{\noindent Eigenvector Configuration $(a)$}

For eigenvector configuration $(a)$, seen in the top left panel of Figure \ref{FIG: EVC_Regions}, there are three regions in which to consider initial conditions:\\

\begin{itemize}
\item \textbf{REGION 1a}: $\ (x_r,y_r)$ on the $x$-axis

\item \textbf{REGION 2a}: $\ (x_r,y_r)$ in the first quadrant below the weak eigenvector $v$ and above the $x$-axis

\item \textbf{REGION 3a}: $\ (x_r,y_r)$ in the first quadrant above the eigenvector $v$\\
\end{itemize}

\noindent Now, we explain how to identify the regions of tolerance given an initial condition $(x_r,y_r)$, using Regions 1a and 2a as examples.\\

\textbf{REGION 1a}: First, we look at the case when the initial condition is on the $x$-axis. In the top left panel of Figure \ref{FIG: AIO_EVC_a}, an arbitrary point on the $x$-axis is shown in the context of eigenvector configuration $(a)$, with lines drawn (portions dashed), showing the addition of scalar multiples of the two eigenvectors to attain the point $(x_r,y_r)$.
We refer to these lines as the $c_{1}$-line and $c_{2}$-line. In this case, they divide the first quadrant into three different subregions, as shown in the top right panel of Figure \ref{FIG: AIO_EVC_a}.

Recall that the P trajectory's initial condition $(x_p,y_p)$
was expressed as $(x_p,y_p)=d_{1}v+d_{2}w$.
For all $(x_p,y_p)$ in a given subregion, there is a corresponding relationship between $d_1,d_2$ and $c_1,c_2$.
Using this relationship, we determine if there exists a region where the criteria $c_1>d_1$ and $c_2<d_2$ of Proposition \ref{THM: Linear Tol 3} and the initial condition criterion $(x_p \geq x_r)$ are all satisfied.
For any $(x_p,y_p)$ in such a region, tolerance will occur, while
for $(x_p, y_p)$ not in such a region, tolerance will not occur .

In fact, for eigenvector configuration $(a)$, if $(x_r,y_r)$ is on the $x$-axis, then there are no subregions in the first quadrant where both $d_1<c_1$ and $d_2>c_2$ hold.
In particular, in $I_{1a}$, $d_1<c_1$ and $d_2<c_2$; in $II_{1a}$, $d_1>c_1$ and $d_2<c_2$; and in $III_{1a}$, $d_1>c_1$ and $d_2>c_2$.
Thus, there exist no $(x_p,y_p)$ that produce tolerance.\\

\textbf{REGION 2a}: Let $(x_r,y_r)$ be in the first quadrant below the weak
eigenvector $v$ (but not on the $x$-axis) in eigenvector configuration $(a)$. The middle left panel of Figure \ref{FIG: AIO_EVC_a} shows an arbitrary point in this region, with lines drawn (portions dashed), showing the addition of the two eigenvectors to attain the point $(x_r,y_r)$.
The middle right panel of Figure \ref{FIG: AIO_EVC_a} shows the four subregions formed in the first quadrant by the $c_1$-line and $c_2$-line. Note that region $\mathbb{IV}_{2a}$ only includes points satisfying $x \geq x_r$. In general, we follow the convention of truncating those subregions that satisfy Proposition \ref{THM: Linear Tol 3} to ensure that $($A3$)$ is satisfied.

In this case, if $(x_p,y_p) \notin \mathbb{IV}_{2a}$, then the conditions of Proposition \ref{THM: Linear Tol 3} fail and tolerance will not occur. In contrast, for $(x_p,y_p)\in \mathbb{IV}_{2a}$, we have that $x_p\geq x_r$ and that $d_1<c_1$ and $d_2>c_2$, such that all of the conditions of Proposition \ref{THM: Linear Tol 3} hold. Hence, for eigenvector configuration $(a)$, if $(x_r,y_r)$ is in the first quadrant below the weak eigenvector $v$ (but not on the $x$-axis), then tolerance will be exhibited precisely for all $(x_p,y_p)\in \mathbb{IV}_{2a}$.\\

\textbf{REGION 3a}: Similarly to the case of Region 2a, the $c_1$-line and
$c_2$-line partition the first quadrant into four subregions, as shown in Figure \ref{FIG: AIO_EVC_a}. The conditions for tolerance only hold in subregion $\mathbb{IV}_{3a}$, which has been truncated to include only points satisfying $x \geq x_r$.

\begin{figure}[h]
\begin{center}
\includegraphics[width=5in]{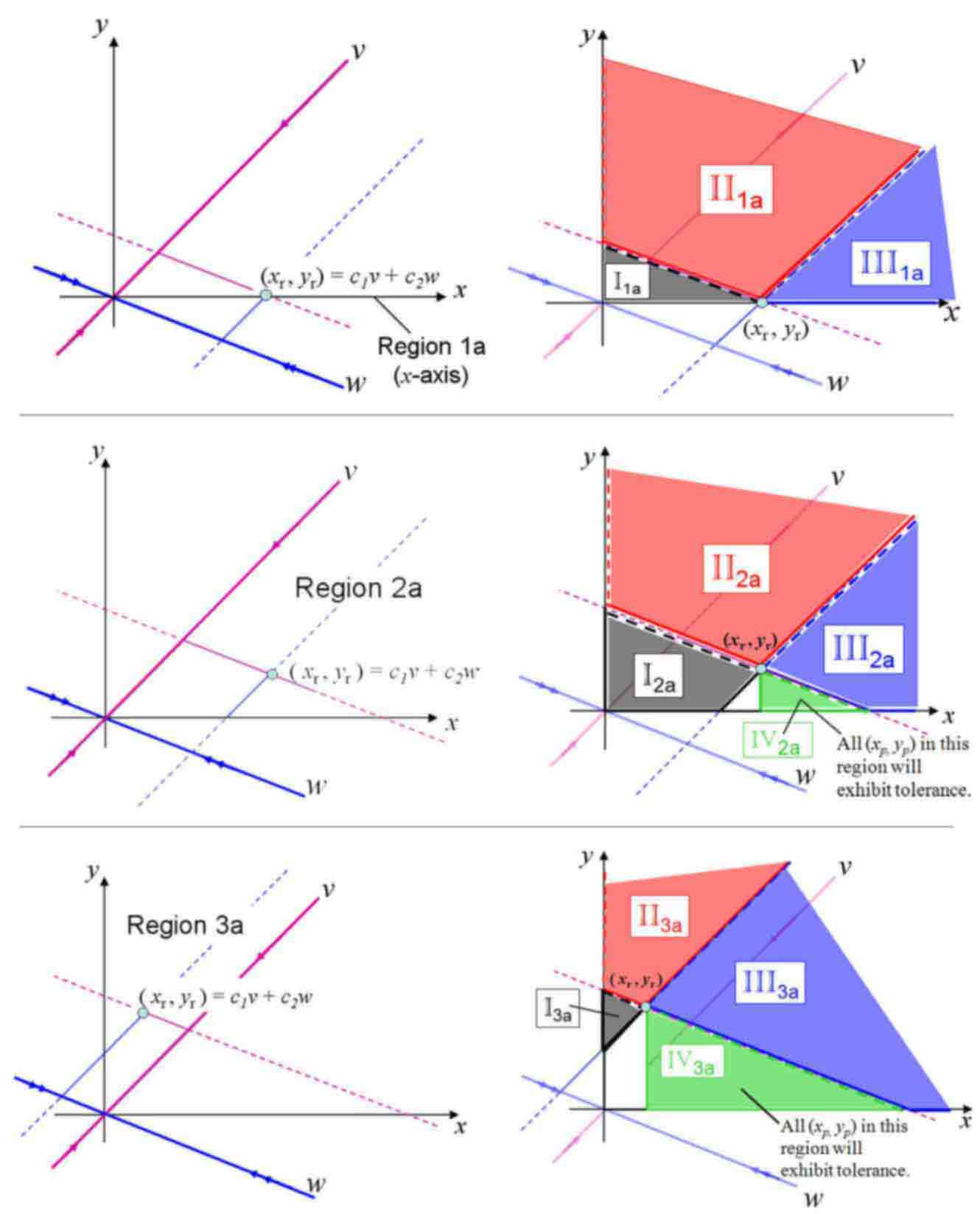}
\caption[Eigenvector configuration $(a)$ with an arbitrary initial condition $(x_r,y_r)$
labeled in various regions.]{Left Side: Eigenvector configuration $(a)$ with an
arbitrary initial condition $(x_r,y_r)$ labeled in Region 1a-3a. Right Side: The first quadrant partitioned into several different subregions by the $c_{1}$-and $c_{2}$-lines associated with the point $(x_r,y_r)=c_{1}v+c_{2}w$
lying in one of the initial regions 1a-3a.}\label{FIG: AIO_EVC_a}
\end{center}
\end{figure}

\begin{figure}[h]
\begin{center}
\includegraphics[width=4in]{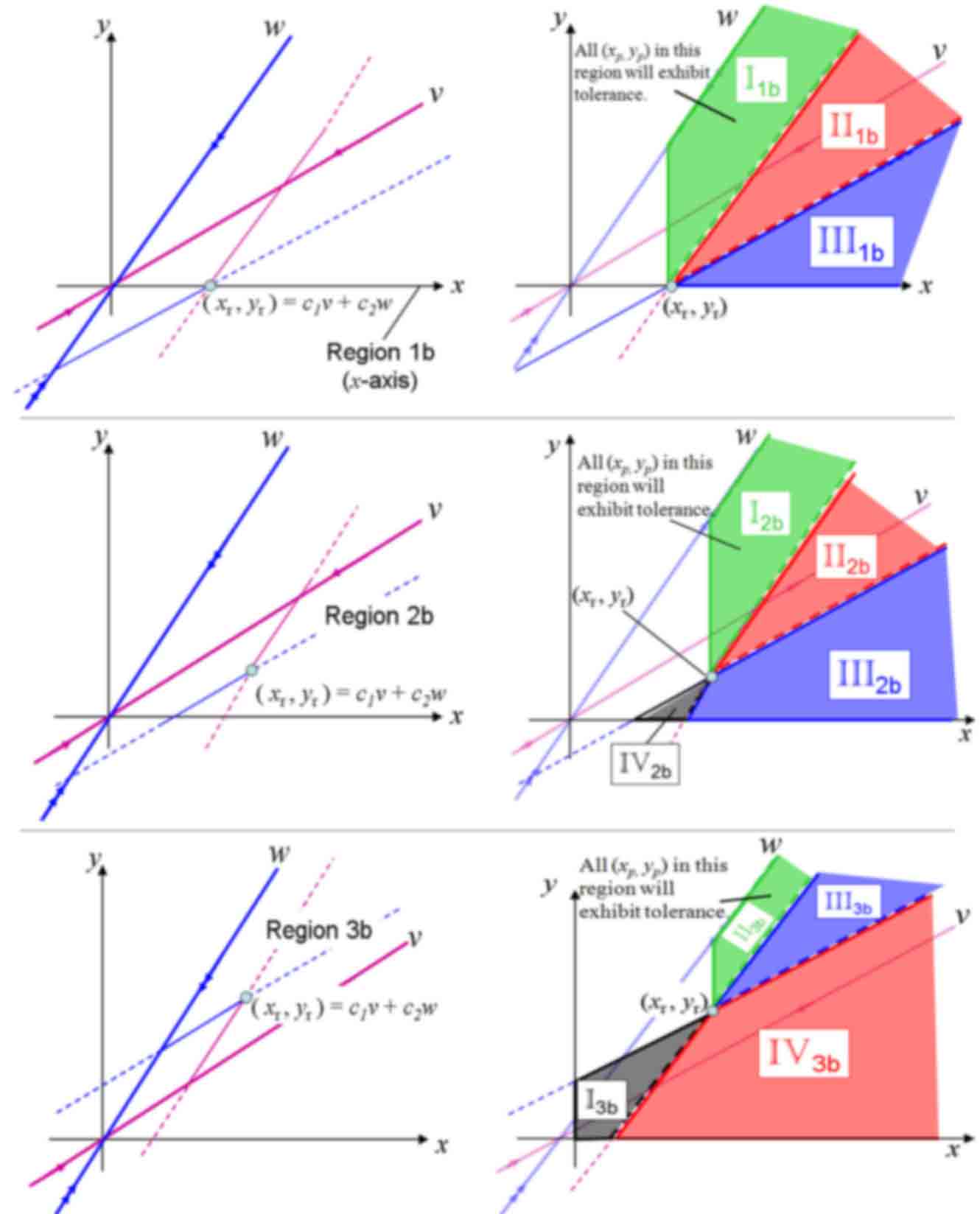}
\caption[Eigenvector configuration $(b)$ with an arbitrary initial condition $(x_r,y_r)$
labeled in various regions.]{Left Side: Eigenvector configuration $(b)$ with an
arbitrary initial condition $(x_r,y_r)$ labeled in Regions 1b-3b. Right Side: The first quadrant partitioned into several different regions by the $c_{1}$-and $c_{2}$-lines associated with the point $(x_r,y_r)=c_{1}v+c_{2}w$ lying in one of the initial regions 1b-3b.}\label{FIG: AIO_EVC_b}
\end{center}
\end{figure}

\pagebreak
\subsubsection{\noindent Eigenvector Configuration $(b)$}
For eigenvector configuration $(b)$, seen in the top right panel of Figure \ref{FIG: EVC_Regions}, there are three regions in which to consider initial conditions:

\begin{itemize}
\item \textbf{REGION 1b}: $\ (x_{c},y_{c})$ on the $x$-axis

\item \textbf{REGION 2b}: $\ (x_{c},y_{c})$ in the first quadrant below the
weak eigenvector $v$ and above the $x$-axis

\item \textbf{REGION 3b}: $\ (x_{c},y_{c})$ in the first quadrant above the weak
 eigenvector $v$ and below the strong eigenvector $w$.
\end{itemize}

\noindent The results for each region are summarized in Table \ref{Tab: EVCResults} and shown in Figure \ref{FIG: AIO_EVC_b}.\\

\subsubsection{\noindent Eigenvector Configuration $(c)$}
\begin{figure}[h]
\begin{center}
\includegraphics[width=4in]{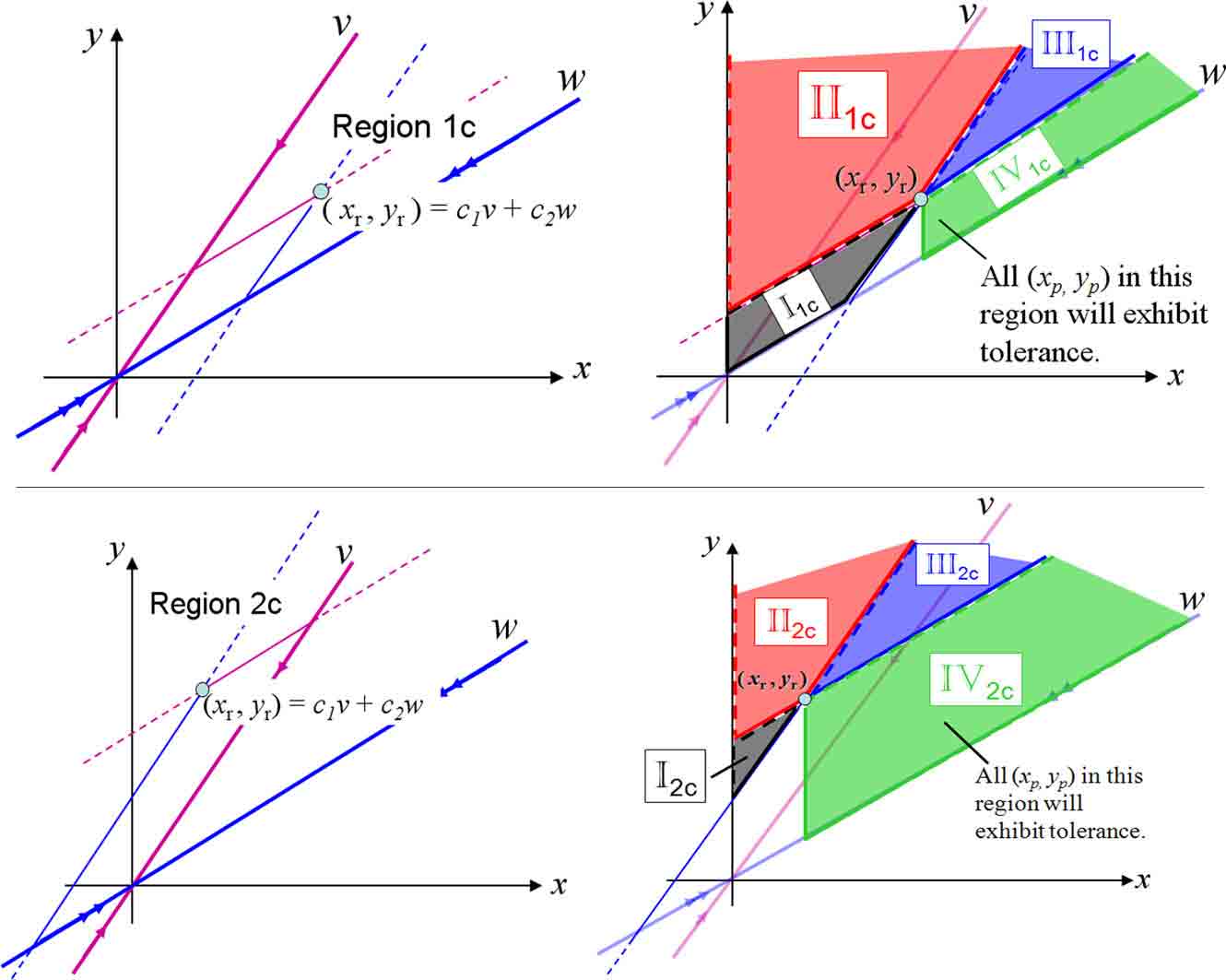}
\caption[Eigenvector configuration $(c)$ with an arbitrary initial condition $(x_r,y_r)$
labeled in various regions.]{Left Side: Eigenvector configuration $(c)$ with an
arbitrary initial condition $(x_r,y_r)$ labeled in Regions 1c-2c. Right Side: The first quadrant partitioned into several different regions by the $c_{1}$-and $c_{2}$-lines associated with the point $(x_r,y_r)=c_{1}v+c_{2}w$ lying in one of the initial regions 1c-2c.}\label{FIG: AIO_EVC_c}
\end{center}
\end{figure}

For eigenvector configuration $(c)$, seen in the bottom left panel of Figure \ref{FIG: EVC_Regions}, there are two regions in which to consider initial conditions:

\begin{itemize}
\item \textbf{REGION 1c}: $(x_r,y_r)$ in the first quadrant below the weak
eigenvector $v$ and above the strong eigenvector $w$

\item \textbf{REGION 2c}: $(x_r,y_r)$ in the first quadrant above both
 eigenvectors
\end{itemize}

\noindent Table \ref{Tab: EVCResults} along with Figure \ref{FIG: AIO_EVC_c} summarize the conclusions about tolerance for the regions in eigenvector configuration (c).\\

\subsubsection{\noindent Eigenvector Configuration $(d)$}
\begin{figure}[h]
\begin{center}
\includegraphics[width=5in]{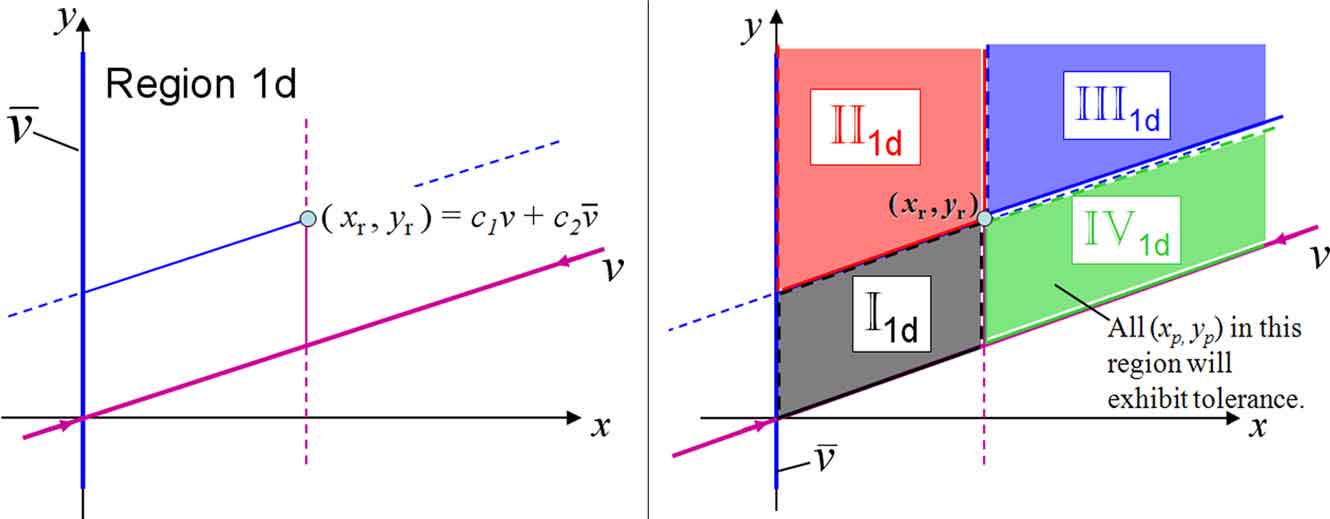}
\caption[Eigenvector configuration $(d)$ with an arbitrary initial condition $(x_r,y_r)$ labeled in Region 1d ($x$-axis).]{Left Panel: Eigenvector configuration $(d)$ with an arbitrary initial condition $(x_r,y_r)$ labeled in Region 1d. Right Panel: The first quadrant of eigenvector configuration $(d)$ partitioned into four subregions by the $c_{1}$- and $c_{2}$-lines associated with the point $(x_r,y_r)=c_1v+c_2\bar{v}$ lying in Region 1d.}
\label{FIG: EVC_a_2bii_2pics}
\end{center}
\end{figure}

To finish our analysis, we examine eigenvector configuration (d), seen in the bottom right panel of Figure \ref{FIG: EVC_Regions}. There is only one region in which to consider initial conditions to explore the existence of tolerance.

\begin{itemize}
\item \textbf{REGION 1d}: $\ (x_r,y_r)$ in the first quadrant above $v$
\end{itemize}

\noindent The conclusion regarding tolerance for this case (Case 2b) was given by Proposition \ref{THM: Case2b}, which shows that it is necessary and sufficient that $\bar{v}_1=0$, $c_1\leq d_1$, and $c_2>d_2$ for tolerance to be exhibited in (\ref{system_linear}). In the left panel of Figure \ref{FIG: EVC_a_2bii_2pics} an arbitrary point in Region 1d is shown in the context of eigenvector configuration $(d)$, with lines drawn (portions dashed), showing the addition of scalar multiples of the eigenvector $v$ and the generalized eigenvector, $\bar{v}$, to attain the point $(x_r,y_r)$. Since $\bar{v}_1=0$ was assumed, the blue line along the $y$-axis represents $\bar{v}$.

The conditions $c_1\leq d_1$ and $c_2>d_2$ are satisfied precisely for those $(x_p,y_p)\in $ $\mathbb{IV}_{1d}$, the region labeled in the right panel of Figure \ref{FIG: EVC_a_2bii_2pics}.
 Moreover, $x_p\geq x_r$ in this region as well. Hence, tolerance will be produced by any $(x_p,y_p)\in \mathbb{IV}_{1d}$, when $(x_r,y_r)$ is in Region 1d under eigenvector configuration (d).

\section{Discussion and Conclusions}
Our consideration of  tolerance serves as an example of how dynamical systems questions can arise from biological phenomena. We initiated our analysis of tolerance under assumptions representative of typical experimental preconditioning protocols used in the study of the acute inflammatory response \cite{Day06,Berg95,Rayhane99,Sly04,Wysocka01}. However, in this paper, we present a generalized analysis, allowing relatively general choices of initial conditions for the reference and perturbed trajectories, since the ideas of inhibition and tolerance, as we have defined them, are themselves quite general. The goal of this analysis is to use information about the initial conditions of the R and P trajectories and the vector field to determine {\em a priori} if the associated trajectories will or will not exhibit tolerance. In tolerance experiments, by applying the challenge dose to the preconditioning trajectory at different times, an experimentalist could generate a continuous curve of possible initial conditions for what we call the P trajectory, and our analysis aims to consider all such initial conditions, to fully characterize the possibility of tolerance within a given experimental set-up.

In the context of two-dimensional nonlinear systems of ODE, it can be difficult to make general statements specifying conditions under which tolerance will be guaranteed to occur. However, our work provides several fundamental statements about configurations of the initial condition $(x_p,y_p)$ for the P trajectory, relative to the R trajectory, that will or will not lead to tolerance. For example, in Section \ref{sub: basic} we have characterized the case when the R trajectory is $n$-excitable, showing that there always exists a subset of the basin of attraction where tolerance is guaranteed to occur for all $(x_p,y_p)$ in the subset. Excitable trajectories are common in systems describing various biological constructs and the idea of tolerance may be important to the ensuing analysis of such systems. By using isoclines and the concept of inhibition, we also present a framework in Section \ref{sub: isoclines} that can be used to derive specific conditions under which tolerance can be ruled out or guaranteed in particular examples.   Techniques such as time interval estimates in Section \ref{Subsection: TimeEstimates} exploit these ideas to achieve a closer examination of transient behavior in the absence of an analytical solution.

In the linear case, we have fully characterized the conditions under which tolerance will or will not occur. A graphical view of the phase plane immediately reveals points $(x_p,y_p)$ that produce tolerance relative to a given $(x_r,y_r)$. For example, Figures \ref{FIG: AIO_EVC_a}-\ref{FIG: EVC_a_2bii_2pics} show regions of $(x_p,y_p)$ (marked in green and labeled) in which tolerance will be exhibited. Interestingly, some of the tolerance regions shown have infinite area (see Figures \ref{FIG: AIO_EVC_b}, \ref{FIG: AIO_EVC_c}, and \ref{FIG: EVC_a_2bii_2pics}). Considering points $(x_p,y_p)$ in the first quadrant and to the right of the vertical line $x=x_r$, we see that in most cases (for instance, see the panels in Figure \ref{FIG: AIO_EVC_b}), the farther $x_p$ is from $x_r$, the higher the $y_p$ value needs to be in order for $(x_p,y_p)$ to fall in the green tolerance region. (As shown using time interval estimates this is also true in nonlinear systems.)  Correspondingly, for some $(x_p,y_p)$ in a tolerance region, tolerance might only occur in the asymptotic limit, which may not be of interest in applications, especially considering that the degree or magnitude of tolerance produced is negligible by then. In other examples (for instance, see the middle and bottom panels of Figure \ref{FIG: AIO_EVC_a}), the $y$-value needs to be sufficiently low for tolerance to occur, although there is a limit on how low it can be because of the non-negativity requirement on $y$.

The issue of tolerance, as defined in this work, does not appear to have received previous analytical treatment. Research has been done on isochronicity, which considers whether multiple phenomena occur within the same interval of time \cite{Sabatini04, Gine05}. For instance, in \cite{Sabatini04}, Sabatini defines a critical point classified as a center to be \textit{isochronous} if every nontrivial cycle within a neighborhood of the critical point has the same period. Although Sabatini noted that the definition of isochronicity does not
require proximity to a critical point, his work and other previous research
appears to have been restricted to locating isochronous sections of autonomous differential systems that are oscillatory in nature \cite{Sabatini04, Gine05, GineLLibre05}.
While tolerance is a natural extension of isochronicity, in that it can be cast in terms of a comparison of the relative passage times of trajectories between sections, previous work has not, to our knowledge, made such comparisons between trajectories converging to a stable node, as we have done here.

Another related area of study is the consideration of phase response curves (PRCs), as are commonly used in the analysis of neuronal systems. PRCs are calculated to determine how instantaneous perturbations shift the phase of a periodic oscillation. Although the assumption of intrinsic oscillatory behavior distinguishes the use of PRCs from the tolerance phenomenon that we consider,
a relation between the two emerges if one thinks of an instantaneous perturbation as a preconditioning event and considers how the subsequent dynamics,  during a specific window of time, compares to the unperturbed oscillation. Depending on where the perturbation occurs in the oscillation cycle, the occurrence of a stereotyped event, such as a peak, can be advanced or delayed relative to the unperturbed case, and the former could be considered as a form of tolerance, in that it would represent a speeding up of the event of interest. Figure \ref{FIG: Voltage} illustrates an example of such a phase advance,  using the Morris-Lecar model. In theory, isoclines could be used to predict whether perturbations in a given system speed up or advance an oscillation. Past work has pointed out that PRCs corresponding to infinitesimal perturbations are intimately related to isochrons, or curves of constant asymptotic phase \cite{winfree}, but these are different than isoclines. Indeed, analysis developed previously for PRCs (see e.g. \cite{bard_mbi} for a review) sheds little light on tolerance under the assumptions that we consider, since there is no intrinsic oscillation involved here. Note that the absence of an oscillation is quite characteristic of the types of models that motivated this work (e.g. \cite{Day06}), since perturbations typically lead to a non-oscillatory decay to a healthy critical point or approach to one or more unhealthy, perhaps lethal, critical points.

\begin{figure}[h]
\begin{center}
  \includegraphics[width=3in]{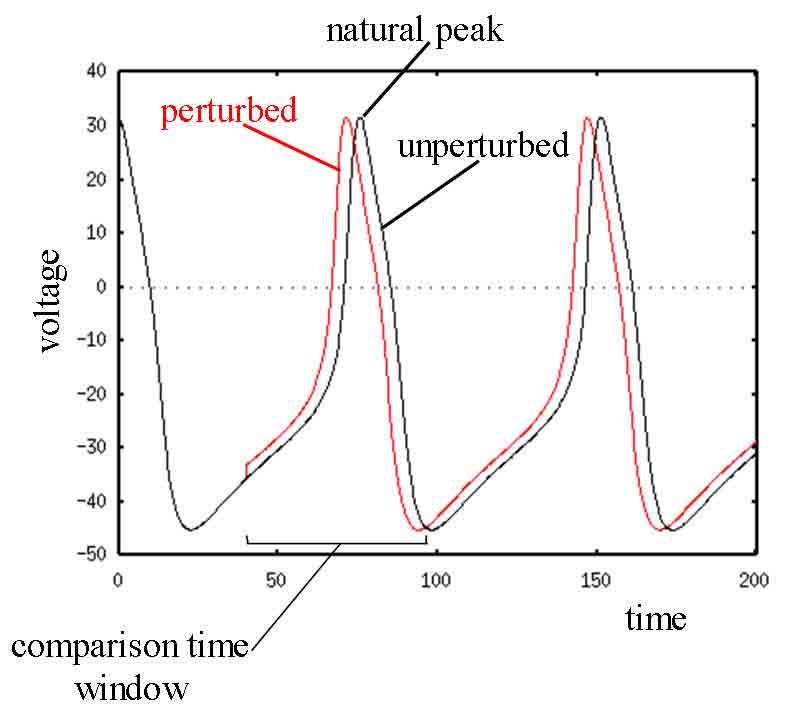}
  \caption[Tolerance in the voltage equation of the Morris-Lecar model]{Tolerance in the voltage equation of the Morris-Lecar model seen during a specific comparison time window.}\label{FIG: Voltage}
\end{center}
\end{figure}

The work presented here looks exclusively at two dimensional ODE systems. Some of the results and techniques considered do not naturally extend to higher dimensions, unfortunately. In \cite{Day06} it was shown that the presence and magnitude of tolerance in a four dimensional ODE model of the acute inflammatory response  depended not only on inhibition but also on the relative levels of the variable being inhibited when various doses of endotoxin were administered, through various feedback effects in the system.
In the $2$D linear case, the relationship between the level of the inhibitory variable and the relative level of the inhibited variable is most clearly seen. Refining the results for the $2$D nonlinear case and extending the results for both linear and nonlinear systems to dimensions greater than two remains to be done. The present work, however, yields new and potentially useful insight into the behavior of transients away from the critical points to which they eventually converge, in the context of some types ODE systems that commonly arise in models of biological systems.

\medskip

{\bf Acknowledgments.}  This work was supported in part by NIH Award  R01-GM67240 (JD, JR), by the Intramural Research Program of
the NIH, NIDDK (CC), and by NSF Awards DMS0414023 (JR), DMS0716936 (JR), and Agreement No. 0635561 (JD).  We thank Gilles Clermont and Yoram Vodovotz for discussions on tolerance in the acute inflammatory response.

\newpage
\bibliographystyle{siam}
\bibliography{Day_Article}

\end{document}